\documentclass[amssymb,amscd,amsfonts,refcheck,11pt,graphicx,verbatim,righttag]{amsart}

\usepackage{amsmath,amscd, amsthm, amsfonts, amssymb,mathrsfs,multirow,graphicx,color}
\usepackage{bm}
\usepackage{amsfonts}
\usepackage{amssymb}
\usepackage[OT2,OT1]{fontenc}
\usepackage{caption}
\usepackage{longtable}
\usepackage[colorlinks,
            linkcolor=blue,
            anchorcolor=blue,
            citecolor=blue
            ]{hyperref}

 \allowdisplaybreaks[4]
\theoremstyle{plain}

\newcommand{\N}{{\mathbb N}}

\newcommand{\Z}{{\mathbb Z}}
\newcommand{\R}{{\mathbb R}}

\newcommand{\Q}{{\mathbb Q}}

\newcommand{\E}{{\mathcal E}}
\newcommand{\B}{{\mathcal B}}

\newcommand{\T}{{\mathbb T}}

\allowdisplaybreaks

\newtheorem{thm}{Theorem}[section]
\newtheorem{cor}[thm]{Corollary}
\newtheorem{prop}[thm]{Proposition}
\newtheorem{rem}[thm]{Remark}
\newtheorem{lem}[thm]{Lemma}
\newtheorem{defi}[thm]{Definition}
\newcommand{\1}{\mathbf{1}}
\newcommand{\2}{\mathbf{2}}

\title{Cantor spectrum for multidimensional quasi-periodic Schr\"odinger operators}
\author{Bernard Helffer}
\address{Laboratoire
de Math\'ematiques Jean Leray,  Nantes Universit\'e  and CNRS, 44 000
Nantes Cedex (France).}
\email{Bernard.Helffer@univ-nantes.fr}

\author{Qinghui LIU}
\address{
Department of Mathematics,
Beijing Institute of Technology,
Beijing 100081, P.R. China.}
\email{qhliu@bit.edu.cn}

\author{Yanhui QU}
\address{Department of Mathematical Science, Tsinghua University, Beijing 100084, P.R. China.}
\email{yhqu@tsinghua.edu.cn}

\author{Qi Zhou}
\address{
Chern Institute of Mathematics and LPMC, Nankai University, Tianjin 300071, P.R. China.
}

 \email{qizhou@nankai.edu.cn}

\begin{document}

\maketitle

\begin{abstract}
In this paper, we investigate the spectrum of a class of multidimensional quasi-periodic Schr\"odinger operators that exhibit a Cantor spectrum, which provides a resolution to a question posed by Damanik, Fillman, and Gorodetski \cite{DFG}. Additionally, we  prove  that for  a dense set of irrational frequencies with positive Hausdorff dimension, the Hausdorff (and upper box) dimension of the spectrum of the critical almost Mathieu operator is positive, yet can be made arbitrarily small.
\end{abstract}

\section{Introduction}

The focus of our research is the spectrum of discrete multidimensional quasi-periodic Schr\"odinger operators on $\ell^2(\Z^d)$, given by
\begin{align}\label{defH}
	(H_{ \lambda V}u)({\bm{n}})=\sum_{|\bm{m}-\bm{n}|=1}u({\bm{m}})+ \lambda  V(\vec{\theta}+  \bm{n} \vec\alpha)u({\bm{n}}),
\end{align}
where $|\bm{m}-\bm{n}|=\sum_{i=1}^d |m_i-n_i|$, $\lambda\in \R$ is the coupling constant.
Here, we assume that the potential $V: \T^d \rightarrow \R$ is continuous, and $\vec{\alpha}=(\alpha_1,\cdots,\alpha_d)$ with $(1,\vec{\alpha})$ rationally independent is the frequency, and denote  $\bm{n} \vec{ \alpha}=(n_1\alpha_1,\cdots,n_d\alpha_d).$
Our main interest lies in determining whether \eqref{defH} possesses Cantor spectrum.

\subsection{Cantor Spectrum for 1D Schr\"odinger operators}

When $ d=1 $, the most prominent example of the operator defined in \eqref{defH} with a Cantor spectrum is the Almost Mathieu Operator (AMO)  $H_{\lambda,\alpha,\theta}$, which is defined on $\ell^2(\mathbb{Z})$ as follows:
\begin{equation}\label{schro}
(H_{\lambda,\alpha,\theta} u)_n = u_{n+1} + u_{n-1} + 2\lambda \cos 2\pi  (n\alpha + \theta) u_n.
\end{equation}
The AMO was originally introduced by Peierls \cite{Pe} as a model for an electron on a two-dimensional lattice subjected to a homogeneous magnetic field \cite{Ha,R}. The famous ``Ten Martini Problem" \cite{Kac,simon} asserts that $ H_{\lambda,\alpha,\theta} $ has a Cantor spectrum for any $ \lambda \neq 0 $ and $ \alpha \in \mathbb{R} \setminus \mathbb{Q} $. This was ultimately proven by Avila and Jitomirskaya \cite{AJ}, with additional contributions from earlier studies \cite{AK06,BS1,CEY,HS3,last3,P}.
Building on methods developed to prove quantitative aspects of Avila’s global theory \cite{avila,GJYZ}, it was recently shown that any small analytic perturbation of AMO retains Cantor spectrum \cite{GJY1,GJY2}.
However, for AMO, it is still open whether ``Dry Ten Martini Problem" (all the spectral gaps are open) holds \cite{AYZ,Kac,simon}. Furthermore, whether ``Dry Ten Martini Problem"  is stable under perturbations, akin to the Cantor spectrum stability demonstrated in \cite{GJY1}, is still widely unresolved.

For more general potentials $ V $, Simon \cite{simon} conjectured that the operator defined in \eqref{defH} generically exhibits a Cantor spectrum. Historically, Eliasson \cite{eliasson:1992} proved that for any fixed Diophantine frequency $ \alpha \in \mathbb{R} \backslash \mathbb{Q} $ and for generic small analytic potentials, \eqref{defH} possesses a Cantor spectrum. Goldstein and Schlag \cite{GS11} established that for generic $ \alpha \in \mathbb{R} \setminus \mathbb{Q} $ in the region of positive Lyapunov exponents, the spectrum forms a Cantor set. Avila-Bochi-Damanik \cite{ABD} showed that for any fixed $ \alpha \in \mathbb{R} \setminus \mathbb{Q} $ and any generic $ V \in C^0(\mathbb{T},\mathbb{R}) $ in the $ C^0 $-topology, the spectrum is also Cantorian \cite{ABD}. In the $ C^k $-topology (for $ 1 \leq k \leq \infty $ or even under the analytic category), it has been proven that for generic $ \alpha \in \mathbb{R} \setminus \mathbb{Q} $ and generic $ V \in C^k(\mathbb{T},\mathbb{R}) $, the spectrum of \eqref{defH} is a Cantor set, for an outline of the proof, we refer to footnote 1 in \cite{WZJ}.

\subsection{ Cantor Spectrum for Multidimensional Schr\"odinger operator}

For $ d \geq 2 $, it is reasonable to anticipate that the Cantor spectrum is a phenomenon exclusive to one dimension. A related conjecture, known as the Bethe-Sommerfeld conjecture, posits that for $ d \geq 2 $ and any periodic function $ V: \mathbb{R}^d \rightarrow \mathbb{R} $, the spectrum of the continuous Schr\"odinger operator $ -\Delta + V $ contains only finitely many gaps; that is, there are no gaps at high energy levels. This conjecture has received extensive studies over the years and was proved  by Parnovski \cite{Pa}. For a more detailed discussion, please refer to \cite{K,Pa} and their references. Additionally, the discrete version of the Bethe-Sommerfeld conjecture has also been resolved recently \cite{EF,HJ}.
More recently, the resolution of the Bethe-Sommerfeld conjecture has been extended to quasiperiodic (not necessarily separable) multidimensional continuous Schr\"odinger operators for almost all frequencies \cite{KPS,KS}.

 In high-energy regimes, continuous Schr\"odinger operators can be treated as perturbations of the free Laplacian. Thus, in the discrete framework, it is plausible to expect that when $ \lambda $ is small, the spectrum of \eqref{defH} forms an interval. This was recently proved by Takase \cite{T} for cases where $ V(\vec\theta) $ is separable (i.e., $ V(\vec\theta) = V_1(\theta_1) + \cdots + V_d(\theta_d) $). These findings elucidate the behavior under conditions of small potential.

However, if $ \lambda $ is large (or if the potential is large), one would still expect the spectrum contains an interval. This expectation is reinforced by the groundbreaking findings of Goldstein-Schlag-Voda \cite{GSV}, who constructed a class of real-analytic functions $ W $ on $ \mathbb{T}^d $ (considered to represent a generic condition) in which, if the frequency is Diophantine and $ \lambda $ is sufficiently large, the spectrum of the one-dimensional operator
\begin{equation}
(H_W u)_n = u_{n+1} + u_{n-1} + \lambda W(\vec\theta+n\vec\alpha) u_n,
\end{equation}
comprises a single interval. This indicates that in high-dimensional cases, if $ V(\vec\theta) $ is degenerate or separable and $ \lambda $ is large, then the spectrum of \eqref{defH} is likely an interval.

Based on these discussions, one might conclude that the Cantor spectrum is predominantly a one-dimensional phenomenon. However, Damanik-Fillman-Gorodetski \cite{DFG} constructed a class of multidimensional limit-periodic Schr\"odinger operators whose spectrum is a Cantor set, spurring interest in identifying a quasiperiodic example. This task is not trivial; the spectral behavior of limit-periodic Schr\"odinger operators diverges significantly from that of quasiperiodic operators \cite{DF}, making the direct application of \cite{DFG}'s methods to quasiperiodic settings challenging. Nonetheless, through Aubry duality \cite{AA80} and Eliasson's results \cite{eliasson:1992}, one can readily construct the long-range operator
\begin{equation}\label{longgeneral}
(Hu)_{\bm{n}} = (Lu)_{\bm{n}} + (Vu)_{\bm{n}} = \sum_{\bm{k} \in \mathbb{Z}^d} V_{\bm{k}} u_{\bm{n} - \bm{k}} + 2\lambda \cos 2\pi(\theta + \langle \bm{n}, \vec\alpha \rangle) u_{\bm{n}},
\end{equation}
which possesses a Cantor spectrum when $ V_{\bm{k}} $ decays exponentially fast. However, the challenge remains in the Schr\"odinger case. Indeed, as expressed by Damanik, Fillman, and Gorodetski \cite{DFG}:

\begin{quote}
``It would be interesting to construct examples of this type having zero-measure Cantor spectrum in which \( L \) is the \( d \)-dimensional discrete Laplacian and \( V \) is multiplication by a suitable quasiperiodic function.''
\end{quote}

In this paper, we aim to address this question by analyzing the multidimensional Almost Mathieu Operator (AMO): \begin{align}\label{defH1} (M_{\lambda \cos, \vec{\alpha}}\, u)({\bm{n}}) = \sum_{|\bm{m}-\bm{n}|=1} u({\bm{m}}) + 2\lambda \left( \sum_{j=1}^{d} \cos 2\pi (\theta_j + n_j \alpha_j) \right) u({\bm{n}}), \end{align} as studied by Bourgain \cite{B}. This operator is also referred to as the multidimensional Aubry-Andr\'e model in the physical literature, and it has attracted  significant interest due to its localization properties, which are found to be substantially more complex compared to the one-dimensional case \cite{BLH,JOD,SS}. Below, we present our main result:

\begin{thm}\label{main}
	For a dense and positive Hausdorff dimension set of $\vec{\alpha}=(\alpha_1,\cdots, \alpha_d) \in \T^d$, the spectrum of $M_{\cos,\vec{\alpha}}$ is a Cantor set of zero Lebesgue measure.
\end{thm}

\begin{rem} Before delving into the  idea of proofs, we provide some remarks regarding Theorem \ref{main}.
	\begin{enumerate}
\item To the best knowledge of the authors, Theorem \ref{main} presents the first example of multidimensional discrete quasiperiodic  Schr\"odinger operators exhibiting Cantor spectra.
	\item If $\lambda \neq 1$, then the spectrum of the operator $H_{\lambda, \alpha,\theta}$ has positive Lebesgue measure \cite{AMS}. Consequently, Steinhaus's Theorem asserts that the spectrum of $M_{\lambda \cos, \vec{\alpha}}$ possesses a dense interior. However, Theorem \ref{main} indicates that this dense interior disappears simultaneously as $\lambda \rightarrow 1$.
\item  Bourgain \cite{B} proved the existence of a positive measure set $\vec{\alpha} \in \mathbb{T}^d$ such that for any $\lambda \neq 0$, the operator $M_{\lambda\cos, \vec{\alpha}}$ possesses spectral gaps. Theorem \ref{main} implies that for a dense subset of $\vec{\alpha} \in \mathbb{T}^d$, the spectral gaps are also dense, if $\lambda=1$.
\item  Takase \cite{T} established that if $\lambda$ is sufficiently small, and $\vec{\alpha}$ is Diophantine,  then the spectrum of $M_{\lambda \cos, \vec{\alpha}}$ constitutes a single interval. This leads to an intriguing question: does there exist a value $\lambda_*$ such that for $\lambda < \lambda_*$, the spectrum forms an interval, while for $\lambda_* < \lambda < 1$, the spectrum becomes a Cantorval? Here, a Cantorval is defined as a nonempty compact subset $\Sigma$ of $\mathbb{R}$ that lacks isolated connected components and contains a dense interior.
		\end{enumerate}
	
\end{rem}

\subsection{Fractal Dimensions of the spectrum of the Critical AMO}

Now we consider the critical AMO:
 \begin{equation}\label{schro-1}
(H_{1,\alpha,\theta} u)_n = u_{n+1} + u_{n-1} + 2 \cos  2\pi(n\alpha + \theta) u_n,
\end{equation}
which is also known as Harper or Azbel-Hofstadter model in physical literature \cite{AOS,Ha}. When $\alpha$ is irrational, the spectrum does not depend on $\theta$ and we denote the spectrum as $\Sigma_{\alpha}$. The spectrum is  beautifully described
via the Hofstadter butterfly \cite{HO}.

Regarding the finer fractal properties of $\Sigma_{\alpha}$, it was widely believed until the mid-1990s that the box dimension $\dim_B\Sigma_{\alpha}$ is equal to $\frac{1}{2}$ for almost all $\alpha$. Numerical and heuristic evidence supporting this conjecture can be found in \cite{conjhalf2, conjhalf3, conjhalf1}. The upper (resp. lower) box dimension, denoted $\overline{\dim}_{B}S$ (resp. $\underline{\dim}_ BS$),
 of a bounded set $S \subseteq \mathbb{R}$ is defined as follows: $$
\overline{\dim}_{B}S = \limsup_{r \downarrow 0} \frac{\log N_r( S)}{\log(r^{-1})}, \qquad \underline{\dim}_{ B} S = \liminf_{r \downarrow 0} \frac{\log N_r(S)}{\log(r^{-1})},
$$ where $N_r(S)$ represents the minimum number of intervals of length $r$ required to cover $S$. When $\overline{\dim}_{ B}S=\underline{\dim}_{ B}S$, we call the common value the box dimension of $S$ and denote by $\dim_B S.$
In 1994, Wilkinson and Austin \cite{wilkinson_austin} provided numerical evidence that $\dim_B\Sigma_{\alpha} = 0.498$ for $\alpha = \frac{\sqrt{5}-1}{2}$, thereby conjecturing that $\overline{\dim}_B\Sigma_{\alpha} < \frac{1}{2}$ for every irrational $\alpha$. However, Jitomirskaya and Zhang \cite{JZ} disproved this conjecture by demonstrating that if $\beta(\alpha) > 0$, then $\overline{\dim}_B\Sigma_{\alpha} = 1$. Here, $\beta(\alpha)$ measures the Liouvillean property of $\alpha$ and is defined as follows:
\begin{equation}\label{defbeta}
\beta(\alpha) := \limsup_{n \rightarrow \infty} \frac{\log q_{n+1}(\alpha)}{q_n(\alpha)},
\end{equation} where $q_n(\alpha)$ denotes the denominator of the $n$-th convergent of $\alpha$.

Considering these results, a natural question arises: \\\textit{If $\beta(\alpha) = 0$, what is the box dimension of the spectrum of the critical AMO?}

In this paper, we will answer  this question as follows:

\begin{thm}\label{main-2}
	For any $\delta\in (0,1)$, there exists a dense and positive Hausdorff dimension set of $\alpha \in \R\backslash \Q$ with $\beta(\alpha)=0$, such that the spectrum $\Sigma_\alpha$ of $H_{1,\alpha,\theta}$ satisfies
	$$0< \dim_H\Sigma_{\alpha}\leq \overline \dim_B \, \Sigma_{\alpha}  \le\delta.$$
\end{thm}
Here the Hausdorff dimension of a set $S \subset \mathbb{R}$ is defined as follows: $$
\dim_H S = \inf \left\{ t \in \mathbb{R^{+}} \, \big{|} \, \lim_{\epsilon \rightarrow 0} \inf_{\epsilon \text{-covers}} \sum_{n} |U_n|^t < \infty \right\},
$$ where a $\epsilon$-cover of $S$ is a family $(U_n)_n$ such that $S \subset \cup_{n=1}^{\infty} U_n$, and each $U_n$ is an interval of length smaller than $\epsilon$; $|A|$ denotes the Lebesgue measure of $A\subset\R$.
Indeed, in recent years, there has been an increasing interest in determining the Hausdorff dimension of $\Sigma_{\alpha}$.
Around 1995, J. Bellissard\footnote{Private conversation with Y. Last.} conjectured that there exists some $\kappa \in (0, 1/2]$ such that $\dim_H \Sigma_{\alpha} = \kappa$ for almost every $\alpha$. Recently, B.~Simon included the problem of determining the Hausdorff dimension of the spectrum of the critical almost Mathieu operator in his new list of significant unsolved problems \cite{BS}.

Let us  review the advancements related to these conjectures and provide comments on Theorem \ref{main-2}. We introduce  the set
$$
\mathcal{L} := \{\alpha \in \mathbb{R} \setminus \mathbb{Q} : \beta(\alpha) > 0\}.
$$
In 1994, Last proved that for a dense subset containing $\mathcal{L}$, the Hausdorff dimension of $\Sigma_{\alpha}$ satisfies $\dim_H \Sigma_{\alpha} \leq \frac{1}{2}$. This conclusion was recently strengthened by Jitomirskaya and Krasovsky \cite{JK} to include all $\alpha$.  Last and Shamis \cite{LS} showed that there exists a dense subset of $\mathcal{L}$ for which $\dim_H \Sigma_{\alpha} = 0$. The results in \cite{LS} were further enhanced by Avila, Last, Shamis, and Zhou \cite{ALSZ}, who established that for any $\alpha \in \mathcal{L}$, $\dim_H \Sigma_{\alpha} = 0$.
However, the box (resp. Hausdorff) dimension may depend sensitively on the arithmetic properties of $\alpha$.
To illustrate this, let us fix any $\delta \in (0,1)$ and define two sets of frequencies as follows:
\begin{equation*}
\mathscr{F}_H(\delta) := \{\alpha \in \mathbb{R} \setminus \mathbb{Q}: \dim_H \Sigma_{\alpha} \leq \delta\}; \quad
\mathscr{F}_B(\delta) := \{\alpha \in \mathbb{R} \setminus \mathbb{Q} : \overline{\dim}_B \Sigma_{\alpha} \leq \delta\}.
\end{equation*}
It is evident that $\mathscr{F}_B(\delta) \subset \mathscr{F}_H(\delta)$, since for any set $A \subset [0,1]$, it holds that $\dim_H A \leq \overline{\dim}_B A$. Moreover, we have $\mathscr{F}_B(\delta) \subsetneq \mathscr{F}_H(\delta)$, as \cite{ALSZ} implies that $\mathcal{L} \subset \mathscr{F}_H(\delta)$, while \cite{JZ} indicates that $\mathscr{F}_B(\delta) \cap \mathcal{L} = \emptyset$.
On the other hand, Helffer-Liu-Qu-Zhou \cite{HLQZ} illustrated the existence of dense frequencies (with positive Hausdorff dimension) in the set $\mathbb{R} \setminus \mathcal{L}$, such that $\dim_H \Sigma_{\alpha} > 0$. Theorem \ref{main-2} is a consequence of the following:
\begin{thm}\label{main-3}
For any $\delta > 0$, the sets $\mathscr{F}_B(\delta)$ and $\mathscr{F}_H(\delta)$ are dense in $\R$ and of positive Hausdorff dimensions.
\end{thm}

\subsection{Ideas of the proof}
The semiclassical analysis of $M_{\cos,\vec{\alpha}}$
was established in \cite{Ra}; however, the estimates provided there are too crude to yield precise spectral information. For multidimensional Schr\"odinger operators with separable potentials as defined in \eqref{defH1}, their spectra can be expressed as Minkowski sums of one-dimensional spectra. In the limit-periodic case, Damanik, Fillman, and Gorodetski \cite{DFG} constructed a Cantor spectrum by analyzing a one-dimensional limit-periodic Schr\"odinger operator with a spectrum of lower box dimension zero, then the  result follows from standard arguments about Minkowski sums of fractal sets. Their methodology, however, fails in our context due to the positivity of the Hausdorff dimension.
To address this, we explicitly demonstrate that the upper box dimension can be made arbitrarily small (Theorem \ref{main-3}), which forces the Minkowski sum to have Lebesgue measure zero (Theorem \ref{main}). The crux of the proof lies in Theorem~\ref{main-3}, whose analysis relies on the semiclassical analysis of the almost Mathieu operator (AMO) developed by Helffer and Sj\"ostrand  \cite{HS1,HS2,HS3}, which will be shortly recalled in the appendices.

\subsubsection{Structure of the spectrum} \label{sec-struc-spec}In our previous work \cite{HLQZ}, we established a lower bound for the Hausdorff dimension of the spectrum for frequencies $\alpha = [a_1, a_2, \ldots]$ where $a_n$ is eventually sufficiently large. Our work is based on  fine covering structure of the spectrum established in \cite{HS1,HS2}.  In \cite{HS1}, Helffer and Sj\"ostrand inductively constructed a family of nested coverings $\{\mathcal{B}_n: n \in \mathbb{N}\}$ of the spectrum. Let us fix a small $\epsilon > 0$. The covering $\mathcal{B}_0$ consists of a single interval (or band) $B_\emptyset$, which is approximately $[-4, 4]$. The covering $\mathcal{B}_1$ is a disjoint family of subintervals of $B_\emptyset$ defined as follows: $$\mathcal{B}_1 = \{B_{-p}, \ldots, B_{-1}, B_0, B_1, \ldots, B_q\},$$ where $B_0 = [-\epsilon, \epsilon]$ and \begin{equation}\label{num-ratio}
p, q \sim a_1; \quad \frac{|B_i|}{|B_\emptyset|} \sim e^{-a_1}, \quad i \neq 0.
\end{equation}
See Figure \ref{fig-black-box} for an illustration. Here, $B_0$ is considered a ``black box" in the sense that complete information about its internal structure is unavailable. The construction of $\mathcal{B}_2$ proceeds as follows: for each $i \neq 0$, there exists a disjoint family of subintervals of $B_i$ given by $$\mathcal{J}_i = \{B_{i(-p_i)}, \ldots, B_{i(-1)}, B_{i0}, B_{i1}, \ldots, B_{iq_i}\},$$ with $B_{i0}$ positioned in the ``middle" of $B_i$ and satisfying $|B_{i0}|/|B_i| \sim \epsilon$, such that \begin{equation}\label{num-ratio-1}
p_i, q_i \sim a_2; \quad \frac{|B_{i\ell}|}{|B_i|} \sim e^{-a_2}, \quad \ell \neq 0.
\end{equation} Again, $B_{i0}$ is treated as a black box. The second-order covering is then defined as $$\mathcal{B}_2 = \{B_0\} \cup \bigcup_{i \neq 0} \mathcal{J}_i.$$ For each interval in $\mathcal{B}_2$ that is not a black box, the above construction can be repeated. By aggregating all resulting intervals, we obtain $\mathcal{B}_3$, and so forth.

\begin{figure}
		\includegraphics[scale=0.7]{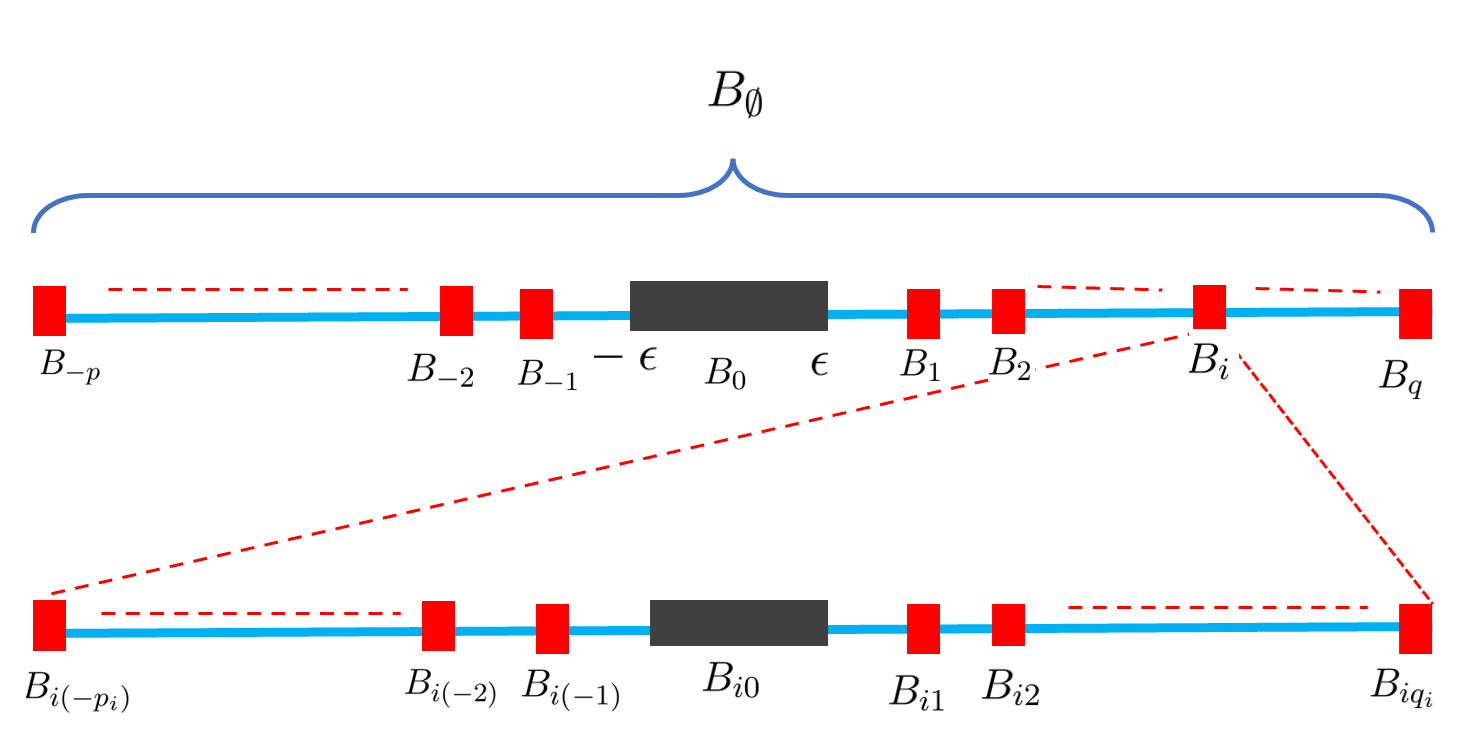}
		
		\caption{The covering $\B_1$ and zoom in of $B_i$}\label{fig-black-box}
 \end{figure}

For each $n$, if we disregard all black boxes in $\mathcal{B}_n$ and denote the new family as $\mathcal{B}_n^*$, we establish a nested covering structure $\{\mathcal{B}_n^*: n \geq 0\}$.
The limit set $X$ of this covering structure satisfies $$X := \bigcap_n \bigcup_{B \in \mathcal{B}_n^*} B \subset \Sigma_\alpha.$$ Using the precise information given by  \eqref{num-ratio} and \eqref{num-ratio-1},
we can estimate the lower bound of the Hausdorff dimension of the spectrum.

However, to derive an upper bound for the upper box dimension of the spectrum, partial information about the structure of the spectrum $\Sigma_\alpha$ is not sufficient; one must understand how the black boxes evolve
during each semiclassical approximation procedure. This is precisely what Helffer and Sj\"ostrand  accomplished in \cite{HS3}.
We will elucidate the situation concerning the first black box $B_0 = [-\epsilon, \epsilon]$, see Figure \ref{fig-B_0} for an illustration. In \cite{HS3}, Helffer and Sj\"ostrand  replaced the black box $B_0$
with a family of subintervals $\{B_0^{(i)}: i \in \mathscr{I}\}$, for which they maintained exact control over $\# \mathscr{I}$ and the ratios $|B_0^{(i)}|/|B_\emptyset|$, similar to \eqref{num-ratio} and \eqref{num-ratio-1}.
More specifically, we introduce  $ h \approx 1/a_1$  and  fix a  constant $M > 1$.   Define
$$\mathscr I_{in}:=\{i\in \mathscr I: B_0^{(i)} \cap [-Mh, Mh]\ne \emptyset\} \ \text{ and }\ \  \mathscr I_{mid}:=\mathscr I\setminus \mathscr I_{in}.$$
Then $0\in \mathscr I_{in}$ and  $B_0^{(0)}$ is  at the center of $B_0$,  of size $h$,  and is  the  largest subinterval. For $i\in \mathscr I_{in}\setminus \{0\}$,
\begin{equation}\label{estimate-in}
    |B_0^{(i)}|\sim -h/\log h \ \ \text{ and }\ \ \#\mathscr I_{in}\sim -\log h.
\end{equation}
  For the bands in $B_0 \setminus [-Mh, Mh]$, the lengths decrease progressively, ranging roughly from $-h/\log h$ to $e^{-1/h}$. The exact description is a bit complex,
  we refer the reader to Definition \ref{config-standard} (vi).  It is in this part, that the bands exhibit typical multi-scale nature.
Consequently, we obtain an upgraded covering
$$\tilde{\mathcal{B}}_1 := \{B_i: i \neq 0\} \cup \{B_0^{(i)}: i\in \mathscr I_{in}\} \cup \{B_0^{(i)}:  i\in \mathscr I_{mid}\}$$
such that for the pair $(B_\emptyset, \tilde{\mathcal{B}}_1)$, we possess complete information regarding the cardinality of $\tilde{\mathcal{B}}_1$ and the ratios $|B|/|B_\emptyset|$ for all $B \in \tilde{\mathcal{B}}_1$. This motivates our definition of a special type of configuration, which will be explicitly detailed in Section \ref{sec-config}.

\begin{figure}
		\includegraphics[scale=0.7]{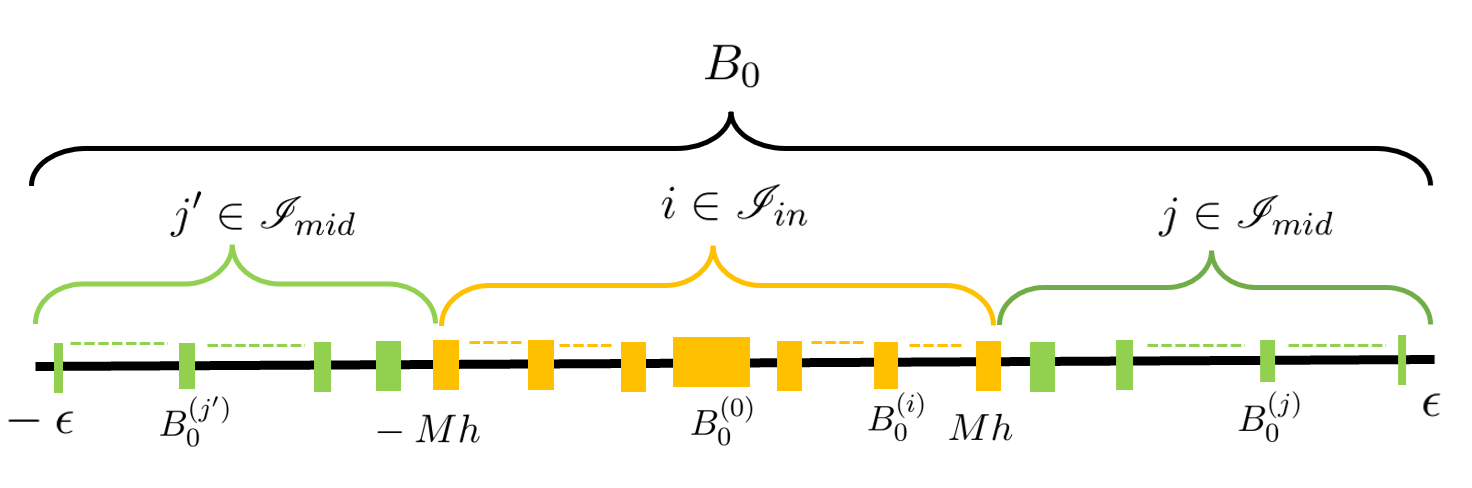}
		
		\caption{The black box $B_0$}\label{fig-B_0}
 \end{figure}

The strength of \cite{HS3} lies in the ability to iterate this process. The authors introduce two types. Specifically, starting from any band $B$ in $\tilde{\mathcal{B}}_1$ with a certain type, one can repeat the aforementioned construction to obtain another pair $(B, \tilde{\mathcal{B}}_B)$, which either behaves  essentially the same as $(B_\emptyset, \tilde{\mathcal{B}}_1)$ if $B$ has type $\1$,  or  can be decomposed into finitely many  (with uniform upper bound) pairs $\{(\hat{B}_j, \tilde{\mathcal{B}}_B^{(j)})\}$, each of which behaves similarly to $(B_\emptyset, \tilde{\mathcal{B}}_1)$ if $B$ has type $\2$ (see also Definitions \ref{defB2} and \ref{defB4}). In this manner, we can construct $\tilde{\B}_2$, ensuring a precise control over the cardinality of $\tilde{\B}_2$ and the ratios of lengths between any band and its sub-bands. This process can be continued to construct the entire family $\{\tilde{\mathcal{B}}_n: n \geq 1\}$. With this family established, we can ultimately estimate the upper bound of the upper box dimension of the spectrum.

\subsubsection{Estimate the fractal dimensions}
In \cite{HS2,HS3}, it was established that the spectrum $\Sigma_\alpha$ is a Cantor set. Furthermore, \cite{HS3} suggested that this result might enable a detailed analysis of the Hausdorff dimension of $\Sigma_\alpha$. The present work realizes this suggestion by providing a concrete study of the Hausdorff dimension, even its box dimension.

To estimate the fractal dimensions of $\Sigma_\alpha$, we analyze the nested covering structure $\{\tilde{\mathcal{B}}_n : n \geq 1\}$. While this structure shares similarities with the Moran sets studied in \cite{FWW}, it exhibits greater complexity. Specifically, the spectrum can be interpreted as a nonlinear, graph-directed, and highly non-homogeneous variant of a Moran set.

Estimating the upper bound for the Hausdorff dimension is relatively straightforward due to the natural covering of the spectrum provided by each $\tilde{\mathcal{B}}_n$. Given the precise metrical information encoded in these coverings, deriving this bound requires only standard arguments. In contrast, the upper box dimension poses significant challenges. For fixed $r > 0$, one has to estimate $N_r(\Sigma_\alpha)$, the minimal number of intervals of length $r$ needed to cover $\Sigma_\alpha$. The difficulty arises from the multi-scale structure of the bands in $\tilde{\mathcal{B}}_B$: the configuration $(B, \tilde{\mathcal{B}}_B)$ exhibits pronounced non-homogeneity in band lengths.

A critical insight resolves this issue: by imposing an upper bound on the coefficients $a_n$, we ensure a lower bound on the length ratio between the shortest  band in $\tilde \B_B$ and $B.$
 This constraint allows us to effectively control $N_r(\Sigma_\alpha)$, thereby bounding the upper box dimension.

Finally, let us explain why the upper box dimension of the spectrum can be made arbitrarily small.  Recall that for a self-similar set $S  $ with open set condition, the Hausdorff and box dimensions are equal and  the common value $s$ is determined through
$\sum_{i=1}^k c_i^s=1$, where $c_i$  is the $i$-th   contraction ratio.  By an analog with this, if for any small $\delta\in (0,1)$ and a typical  configuration like $(B_\emptyset, \tilde \B_1)$, we have
$$\sum_{B\in \tilde \B_1} r_B^\delta\le 1, \ \ \text{ where }\ \ r_B=\frac{|B|}{|B_\emptyset|},$$
then we can use similar argument to show that  the Hausdorff dimension of the spectrum is less than $\delta$. If in addition, $\{a_n:n\in \N\}$ is bounded from above, then the upper box dimension of the spectrum is also less than $\delta$. By the definition of $\tilde \B_1$, we can write
$$\sum_{B\in \tilde \B_1} r_B^\delta=\sum_{i\ne 0}\left(\frac{|B_i|}{|B_\emptyset|}\right)^\delta+\sum_{i\in \mathscr I_{in}}\left(\frac{|B_0^{(i)}|}{|B_\emptyset|}\right)^\delta+\sum_{i\in \mathscr I_{mid}}\left(\frac{|B_0^{(i)}|}{|B_\emptyset|}\right)^\delta=: S_{out}(\delta)+S_{in}(\delta)+S_{mid}(\delta).$$
By \eqref{num-ratio} and \eqref{estimate-in}, it is seen that for $a_1$ large enough (equivalent, for $h$ small enough), we have  $S_{out}(\delta)<1/3$ and $S_{in }(\delta)<1/3$.  But for $S_{mid}(\delta)$, there is no easy path to estimate. We need to decompose the sum according to the different scales and estimate them separately.  Although quite tricky, it turns out one can choose suitable $h$ such that $S_{mid}(\delta)<1/3.$ Consequently we get the desirable estimates.

The above argument assumes $a_n$ is large for all $n \geq 1$. To extend these fractal dimension estimates to a dense set of frequencies $\alpha$, we relax this condition: it suffices for them that there exists   $m \in \mathbb{N}$ such that $a_n$ is sufficiently large for all $n \geq m$. To achieve this, we need the full strength of  \cite{HS2} and to extend slightly \cite{HS3}.  For details, we refer the reader to   Appendix~\ref{appc}.


\subsection{Outline of the paper}

The rest of the paper is organized as follows. In Section \ref{sec-proof-main},  we  prove Theorem \ref{main}   by assuming Theorem \ref{main-3}. In Section \ref{sec-config}, we define certain  abstract configurations and study their basic properties.  In Section \ref{sec-nested-covering}, we define an abstract  nested covering structure  and obtain the upper bound estimates of the Hausdorff and upper box dimensions  of the limit set.  In Section \ref{sec-dim}, we apply the result established in Section \ref{sec-nested-covering} to critical AMO and prove  Theorem \ref{main-3} and Theorem \ref{main-2}.   In Appendices  \ref{sec-ap-a},  \ref{sec-ap-b} and \ref{appc}
 we extract from \cite{HS1,HS2,HS3} what we need  for the description of the nested covering structure.

\section{Proof of Theorem \ref{main}  }\label{sec-proof-main}

  In this section, we prove Theorem \ref{main}, which is a  direct consequence of  Theorem \ref{main-3}. The proof of Theorem \ref{main-3} is considerably more involved and will be given in Section \ref{sec-dim}.

Let us recall some known facts from fractal geometry.

\begin{lem}[\cite{Fal}]\label{fractal-basic}
  Assume $A_1,\cdots, A_d\subset \R$ are  nonempty. Then

  (1)  $\sum_{i=1}^{d}\dim_H A_i\le \dim_H \Big(\prod_{i=1}^{d}A_i\Big).$

  (2) If $A_i, 1\le i\le d$  are bounded, then
  $$\overline\dim_B (A_1+\cdots+A_d)\le \overline\dim_B \Big( \prod_{i=1}^{d}A_i\Big)\le \sum_{i=1}^{d}\overline\dim_B A_i.$$
\end{lem}

\begin{proof}
  (1) See \cite{Fal} Product formula 7.2 .

  (2)
  See \cite{Fal} Product formula 7.5 for the second inequality.  Define the map $P:\R^d\to \R$  as $P(x_1,\cdots,x_d):=\sum_{i=1}^{d}x_i$. Then $P$ is Lipschitz.  By  \cite{Fal} Sec. 3.2 iv), we have
  $$
  \overline\dim_B (A_1+\cdots+A_d)=\overline\dim_B \Big(P(\prod_{i=1}^{d}A_i)\Big)\le \overline\dim_B \Big( \prod_{i=1}^{d}A_i\Big).$$
  So (2) holds.
\end{proof}

 In the following discussion, when considering any self-adjoint operator $ H $ defined on a Hilbert space $ \mathcal{H} $ ($ \mathcal{H}$ usually denotes  $\ell^2(\mathbb{Z}^d) $ or $ L^2(\mathbb{R}) $), we denote the spectrum of $ H $ by $ \mathrm{Sp}(H) $ and continue to use $ \Sigma_{\alpha} $ to represent the spectrum of the critical  AMO.

\noindent {\bf Proof of Theorem \ref{main}.}
Fix $\delta=1/(d+1)$. Define
$$
\mathcal F:=\{\vec \alpha\in \R^d: |{\rm Sp}(M_{\cos,\vec \alpha})|=0\}.
$$

We claim that: $\mathscr F_B(\delta)^d\subset \mathcal F.$
Indeed, for any $\vec \alpha=(\alpha_1,\cdots,\alpha_d)\in \mathscr F_B(\delta)^d,$ by \cite{DG} Proposition 6.1 (a), we have
$$
{\rm Sp}(M_{\cos,\vec \alpha})=\Sigma_{\alpha_1}+\cdots+\Sigma_{\alpha_d}.
$$
Since $\alpha_i\in \mathscr F_B(\delta)$, by Lemma \ref{fractal-basic}(2),
$$
\overline\dim_B \Big({\rm Sp}(M_{\cos,\vec {\alpha}})\Big) =\overline\dim_B\big(\Sigma_{\alpha_1}+\cdots+\Sigma_{\alpha_d}\big)\le d\delta<1.
$$
Consequently $|{\rm Sp}(M_{\cos,\vec \alpha})|=0.$
So the claim holds.

By Theorem \ref{main-3},  $\mathscr F_B(\delta)$ is dense in $\R$ and $\dim_H \mathscr F_B(\delta)>0$.  So we conclude that $\mathscr F_B(\delta)^d$
is dense in $\R^d$.  By Lemma \ref{fractal-basic}(1), we have
$\dim_H \Big( \mathscr F_B(\delta)^d\Big) \ge d \dim_H\mathscr F_B(\delta)>0.$
Then by the claim, $\mathcal F$ is dense in $\R^d$ and has positive Hausdorff dimension. So $\mathcal F+\Z^d\subset \T^d$ is dense in $\T^d$ and has positive Hausdorff dimension.
\hfill $\Box$

\section{ Configurations and their  basic properties}\label{sec-config}

In this section, we define several configurations related to  an interval and a family of its subintervals.    They are  abstractions  from the covering structure of the spectrum via one step of semiclassical approximation, as we explained in Paragraph \ref{sec-struc-spec}. We also study its basic properties, which will be useful later for estimating the dimensions of the spectrum.

\subsection{Primitive  configuration}\

At first, we introduce some notation.

Given two intervals  $I=[a,b]$ and $J=[c,d]$, we write $I<J$ if $b<c.$
Assuming that  $\mathcal J$ is a finite family of compact intervals, we write:
  \begin{equation*}
    \mathcal J_{\min}:=\min\{|J|: J\in \mathcal J\}\ \ \text{ and }\ \ \mathcal J_{\max}:=\max\{|J|: J\in \mathcal J\}.
  \end{equation*}

\begin{defi}
  If  $I$ is a compact interval  and    $\mathcal J$  is   a  finite  class  of  disjoint  subintervals of $I$,  we call $(I,\mathcal J)$ a configuration.   If moreover,  $\#\mathcal J\ge 3$ and we index $\mathcal J$ as
$$ \mathcal J:=\{J_i: -r\le i\le s\}$$
  such that
$J_i< J_{i+1}$ for any $-r\le i<s$, then we call $(I,\mathcal J)$ a $[r,s]$-configuration.
\end{defi}

  Assume $(I,\mathcal J)$ is a $[r,s]$-configuration.  For any $1\le i\le s$, we denote the gap between $J_{i-1}$ and $J_i$ by $G_i$. For any $1\le i\le r$, we denote the gap between $J_{-i}$ and $J_{-i+1}$ by $G_{-i}$. See Figure \ref{fig-r-s-config}  for an illustration.

\begin{figure}
		\includegraphics[scale=0.7]{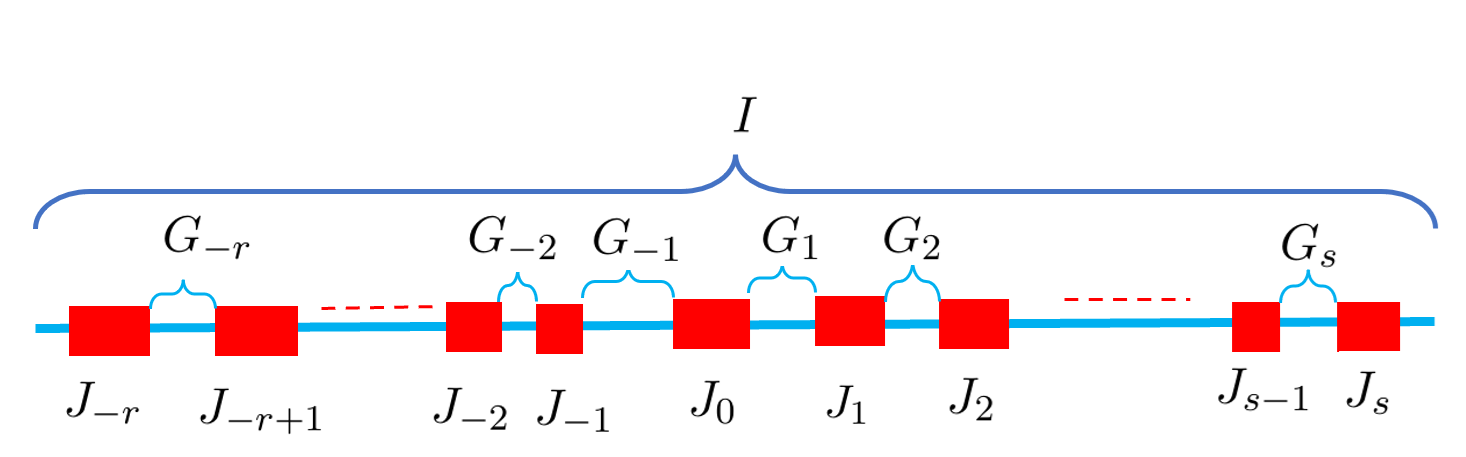}
		
		\caption{A $[r,s]$-configuration}\label{fig-r-s-config}
 \end{figure}

  \subsection{Configurations with metrical control} \

  The following definition is strongly motivated by the pair $(B_\emptyset, \tilde \B_1)$  mentioned  in  Paragraph \ref{sec-struc-spec}. See Figure \ref{fig-standard-config} for an illustration of the definition.

\begin{defi}\label{config-standard}
Given a $[r,s]$-configuration $(I,\mathcal J)$.  Write
 $\mathscr I:=\{i\in\Z:-r\le i\le s\}$ and
$$
I:=[\eta,\xi];\ \  J_i:=[\eta_i,\xi_i],\  (i\in \mathscr I).
$$
  Given the following parameters:
\begin{equation}\label{parameter}
 \varsigma\in (0,4), \ \epsilon\in (0,\varsigma/100),\  M> C>1,\ 0<h<(1/C)\wedge(\epsilon/ M)\wedge e^{-1/C}.
\end{equation}
$(I,\mathcal J)$ is called a standard $(\varsigma,\epsilon, M, C,h)$-configuration if the following hold:

(i)   $I$ satisfies
$$[-\varsigma,\varsigma]\subset I=[\eta,\xi]\subset [-4,4].$$

(ii)  The numbers $r,s$ satisfy
 $$C^{-1}h^{-1}\le r,s\le Ch^{-1}.$$

 (iii) The ``central", the most left and the most right subintervals satisfy
$$0\in J_0,\ \  C^{-1}h\le |J_0|\le Ch;\ \ \eta_{-r}=\eta;\ \ \xi_s=\xi.$$

(iv) Write $r_1:=-\min \mathscr I_{in}, s_1=\max \mathscr I_{in}$, where
$$
\mathscr I_{in}:=\{i: J_i\cap [-Mh,Mh]\ne\emptyset\}.
$$
we have
$$-C^{-1}\log h\le r_1,s_1\le -C\log h.$$
Moreover, for any $i\in \mathscr I_{in}\setminus \{0\}$, we have
$$
\frac{h}{-C\log h}\le|J_i|  \le \frac{Ch}{-\log h}\ \ \text{ and }\ \
\frac{h}{-C\log h}\le  |G_i| \le C h.
$$

(v)  Write  $\mathscr I_{out}:=\mathscr I_{out}^-\bigsqcup \mathscr I_{out}^+$, where
$$
\mathscr I_{out}^-:=\{i: J_i\cap [\eta,-\epsilon]\ne\emptyset\};\ \ \ \mathscr I_{out}^+:=\{i: J_i\cap [\epsilon,\xi]\ne\emptyset\}.
$$
Then for any $i\in \mathscr I_{out}$, we have
$$
e^{-\frac{C}{h}}\le |J_i|\le e^{-\frac{1}{Ch}};\ \ C^{-1}h\le |G_i|\le Ch.
$$

(vi) For any $i\in \mathscr I_{mid}:=\mathscr I\setminus(\mathscr I_{in}\cup
 \mathscr I_{out})$, we have
$$
 e^{-\frac{|c_i|C}{h}}\frac{h}{-C\log|c_i|}\le |J_i|\le  e^{-\frac{|c_i|}{Ch}}\frac{Ch}{-\log|c_i|};\ \  \frac{h}{-C\log |g_i|}\le |G_i|\le \frac{Ch}{-\log |g_i|},
$$
where $c_i$ and $g_i$ are the centers of $J_i$ and $G_i$, respectively.

\end{defi}
See Figures \ref{fig-inside}, \ref{fig-outside}, \ref{fig-middle} for geometric illustrations of  $(iv), (v), (vi),$ respectively.

\begin{figure}
		\includegraphics[scale=0.7]{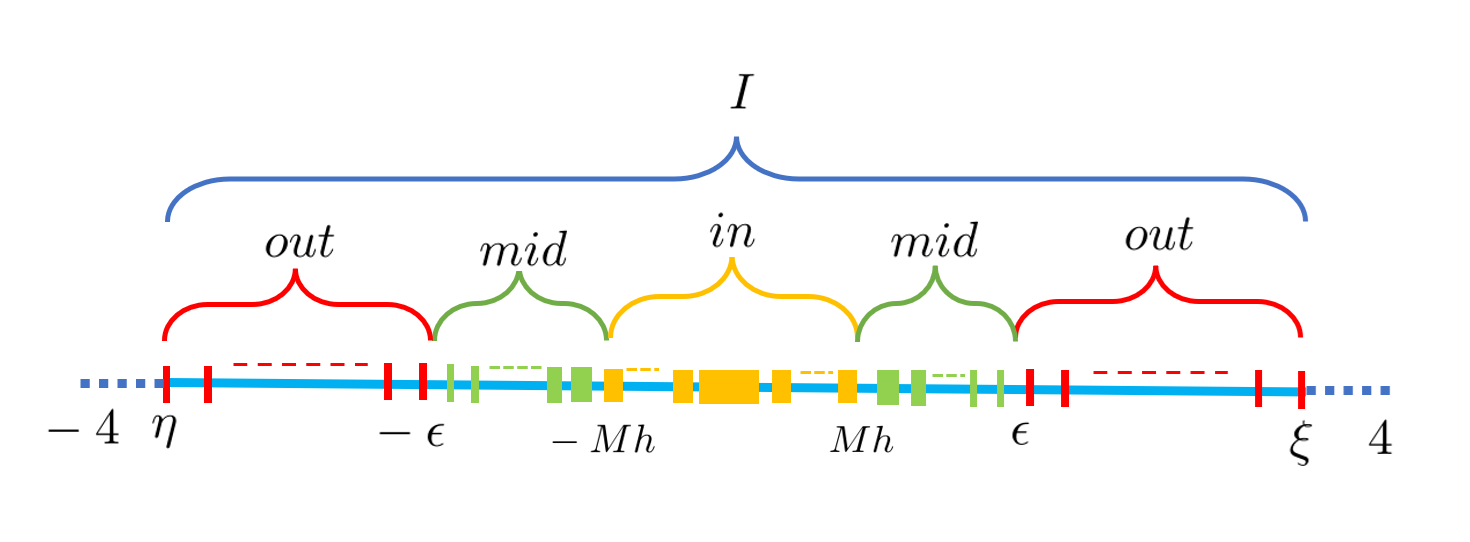}
		
		\caption{A  standard configuration from far away}\label{fig-standard-config}
 \end{figure}

\begin{figure}
		\includegraphics[scale=0.7]{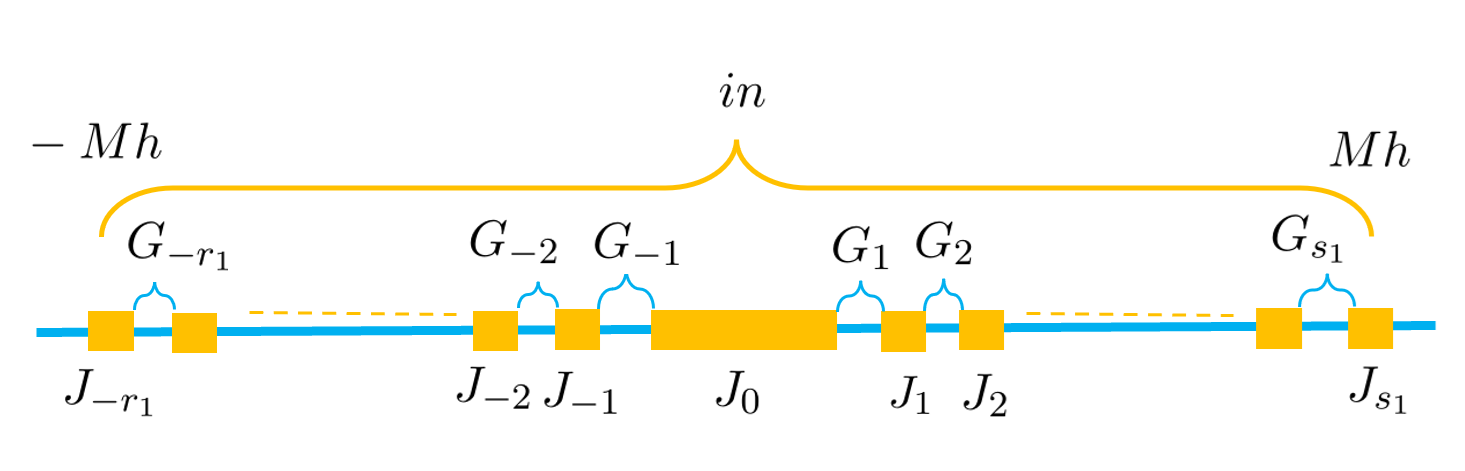}
		
		\caption{zoom in of the inside part}\label{fig-inside}
 \end{figure}

\begin{figure}
		\includegraphics[scale=0.8]{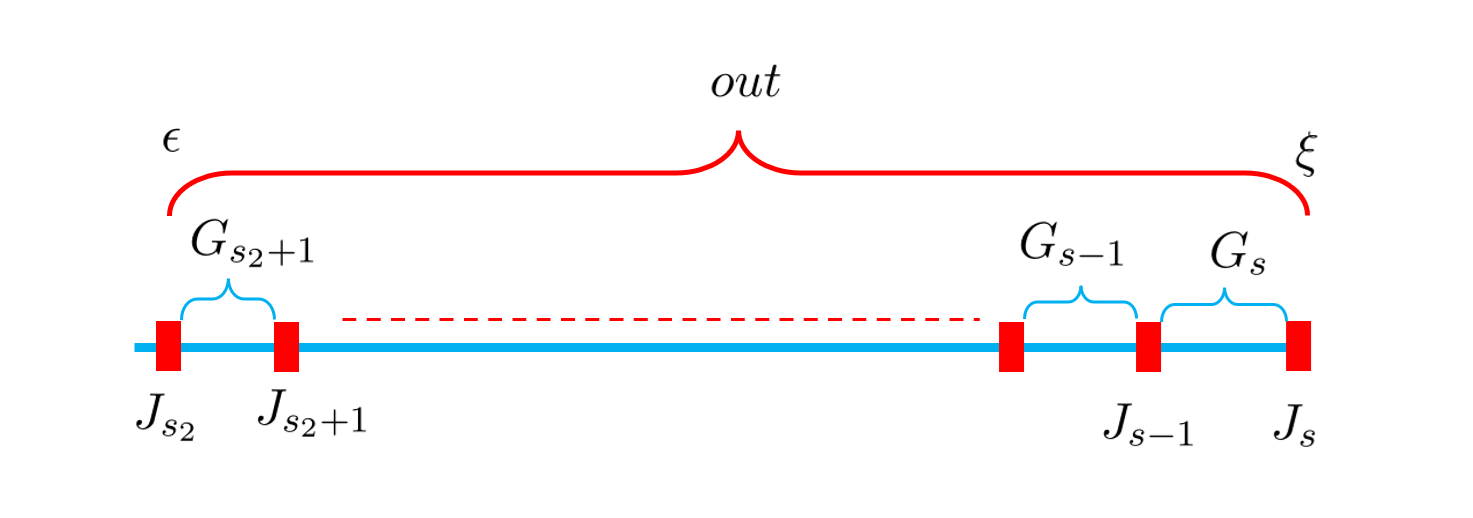}
		
		\caption{zoom in of the right  outside part, where $s_2=\min \mathscr I_{out}^+$}\label{fig-outside}
 \end{figure}

\begin{figure}
		\includegraphics[scale=0.8]{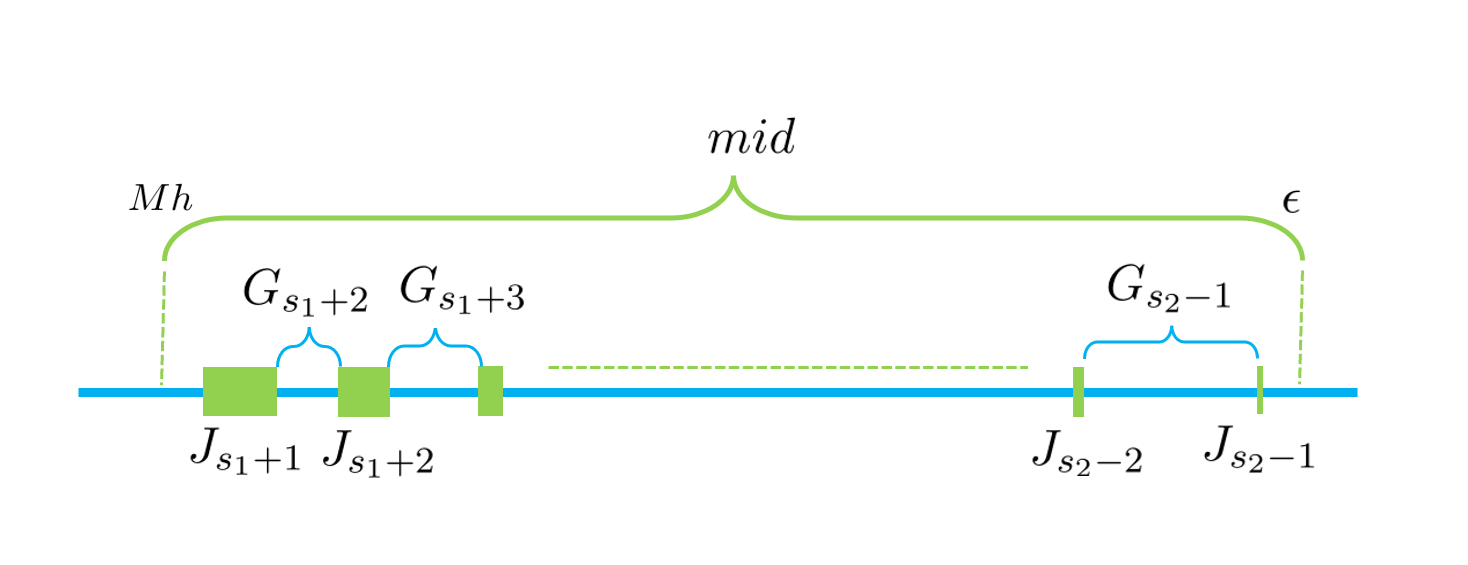}
		
		\caption{zoom in of the right middle part}\label{fig-middle}
 \end{figure}
Let us first make some simple observations that will be useful later.

\begin{lem}\label{lenth-of-band}
 Assume $\varsigma,\epsilon, M,C$ satisfy \eqref{parameter}. Then there exists $\hat h=\hat h(C)>0$ such that if $h\in (0,\hat h]$ and  $(I,\mathcal J) $ is a standard $(\varsigma,\epsilon, M,C,h)$-configuration, then
   \begin{equation}\label{upper-lower-length}
    e^{-C/h}\le \mathcal J_{\min}\le \mathcal J_{\max}\le Ch.
   \end{equation}

\end{lem}

\begin{proof}
  Choose $\hat h=\hat h(C)>0$ such that for any $h\in (0,\hat h]$, we have
\begin{equation}\label{h_0}
 e^{-\frac{1}{Ch}}\le Ch \ \ \text{ and }\ \  e^{-\frac{C}{h}}\le e^{-\frac{C}{10h}}\frac{h}{-C\log h} .
\end{equation}
Notice that if $i\in \mathscr I_{mid}$, then
$$h\le Mh\le|c_i|\le \epsilon\le \varsigma/100<1/10.$$
So if $h\in (0,\hat h]$, by Definition \ref{config-standard} and \eqref{h_0},
\begin{eqnarray*}
\mathcal J_{\max}&\le& \max\left\{Ch, \frac{Ch}{-\log h}, e^{-\frac{1}{Ch}}, \frac{Ch}{\log 10}\right\}\le Ch,\\
\mathcal J_{\min}&\ge& \min\left\{\frac{h}{C}, \frac{h}{-C\log h}, e^{-\frac{C}{h}}, e^{-\frac{C}{10h}}\frac{h}{-C\log h}\right\}\ge e^{-\frac{C}{h}}.
\end{eqnarray*}
So \eqref{upper-lower-length} holds.
\end{proof}

To model the more general pair $(B,\tilde \B_B)$  which  is mentioned  in Paragraph \ref{sec-struc-spec},  we introduce two more configurations as follows:

\begin{defi}
  Assume $(I,\mathcal J)$ is a $[r,s]$-configuration and  $\varsigma,\epsilon, M,C, h$ satisfy \eqref{parameter}. We call $(I,\mathcal J)$ an $(\varsigma,\epsilon, M,C,h)$-configuration if there exists an affine map $T$ such that $(T(I),T(\mathcal J))$ is a standard $(\varsigma,\epsilon, M, C,h)$-configuration, where
  $
  T(\mathcal J):=\{T(J): J\in \mathcal J\}.
  $
\end{defi}


\begin{defi}\label{config-kappa-rho}
    Assume $(I,\mathcal J)$ is a  configuration and  $\varsigma,\epsilon, M,C, h$ satisfy \eqref{parameter}.  Assume $k\in \N$ and $\rho\in (0,1)$. We call $(I,\mathcal J)$ a  $(k,\rho;\varsigma,\epsilon, M, C,h)$-configuration if the following hold:

(i)
 There exists a  disjoint family  $\{I_1,\cdots, I_k\}$ of subintervals of $I$   such that
     $$
    \frac{\rho}{k}\le  \frac{|I_i|}{|I|}\le \frac{1}{\rho k}, \  \forall i=1,\cdots,k
     $$
     and for any $J\in \mathcal J$, there exists some $i$ (hence unique) such that $J\subset I_i$.

  (ii)    For any $i$, define
    $
    \mathcal J_i:=\{J\in \mathcal J: J\subset I_i\}.
    $ Then
     $(I_i,\mathcal J_i)$ is a $(\varsigma,\epsilon, M,C,h)$-configuration.

\end{defi}
See  Figure \ref{fig-k-rho-config} for an example of a  $(3,\rho;\varsigma,\epsilon, M, C,h)$-configuration.

\begin{figure}
		\includegraphics[scale=0.8]{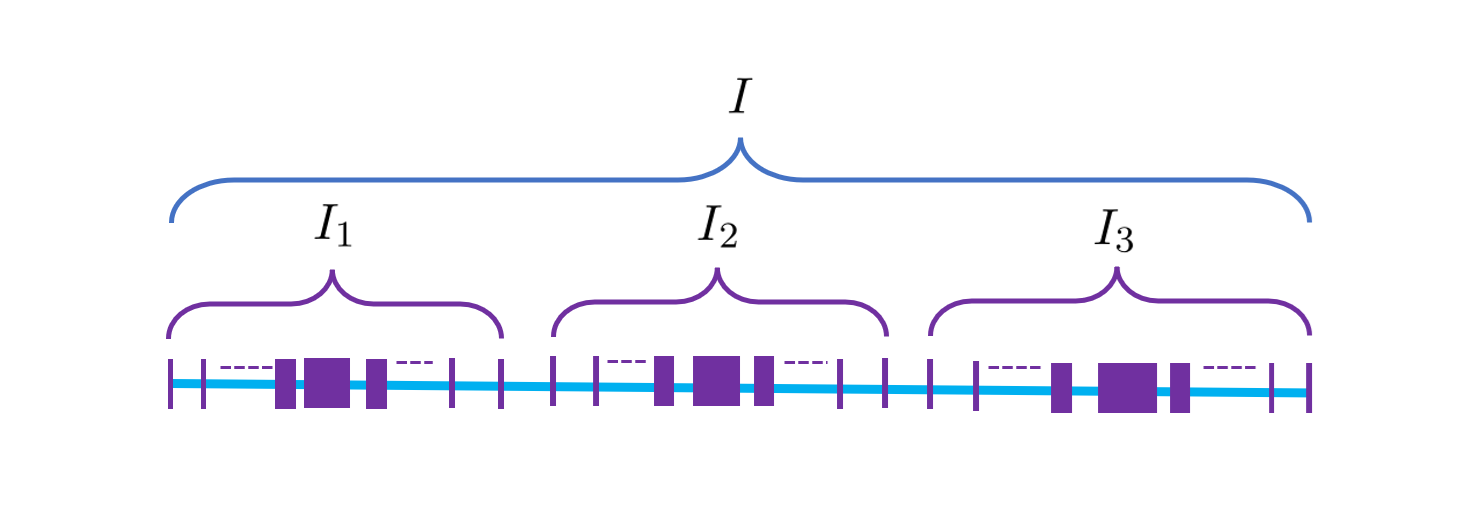}
		
		\caption{A $(3,\rho;\varsigma,\epsilon, M, C,h)$-configuration}\label{fig-k-rho-config}
 \end{figure}
\begin{rem}\label{type2f}

  (i) We call the family $\{(I_i,\mathcal J_i):1\le i\le k\}$ the sub-$(\varsigma,\epsilon, M,C,h)$-configurations of $(I,\mathcal J).$

  (ii)
  If $k=1$, then  a $(k,\rho;\varsigma,\epsilon, M,C,h)$-configuration is just a $(\varsigma,\epsilon, M,C,h)$-configuration.

\end{rem}

We have the following observation, which is crucial later for estimating the dimensions of the spectrum.

\begin{lem}\label{pre-dim-delta}
Given $k\in \N$, $\rho\in (0,1)$. Assume $\varsigma,\epsilon, M, C, h$ satisfy \eqref{parameter}. Assume  $(I,\mathcal J)$ is a $(k,\rho; \varsigma,\epsilon,M,C,h)$-configuration. Let $\hat h$ be the constant in Lemma \ref{lenth-of-band}.  Then

(1)  There exists $\tilde h=\tilde h(C,\varsigma,\rho)\in (0,\hat h]$ such that if $h\in (0,\tilde h] $, then for any $J\in \mathcal J$,
\begin{equation*}
  \frac{\rho}{8ke^{C/h}}\le \frac{|J|}{|I|}\le \frac{Ch}{2k\rho \varsigma}\le \frac{1}{10}.
\end{equation*}

(2) Assume $h\in (0,\tilde h]$. Assume $r>0$ satisfies
$\mathcal J_{\min}\le r<|I|$. Then
$$N_r(I)\le \frac{16k e^{C/h}}{\rho}.$$

(3) Fix $\kappa\in \N.$ Then for any $\delta\in(0,1)$  there exists $h(\delta)\in (0,\tilde h]$ such that if $h\in (0,h(\delta)]$ and $1\le k\le \kappa$,   then
\begin{equation}\label{delta}
\sum_{J\in \mathcal J}\left(\frac{|J|}{|I|}\right)^\delta\le 1.
\end{equation}

\end{lem}

\begin{proof}
(1) Define $\tilde h$ as
\begin{equation}\label{def-tilde-h}
  \tilde h:=\min\left\{\hat h, \frac{2\rho\varsigma}{10C}\right\}.
\end{equation}

Fix any $h\in(0,\tilde h]$. Assume $J_\ast\in \mathcal J_i$ is such that $|J_\ast|=\mathcal J_{\min}.$
By Definition \ref{config-kappa-rho}, $(I_i,\mathcal J_i)$ is a $( \varsigma,\epsilon,M,C,h)$-configuration and $|I_i|/|I|\ge \rho/k$.  So by Lemma \ref{lenth-of-band},
\begin{align*}
  \frac{\mathcal J_{\min}}{|I|} & =\frac{|J_\ast|}{|I_i|}\frac{|I_i|}{|I|}\ge\frac{e^{-C/h}}{8}\frac{\rho}{k}=\frac{\rho}{8ke^{C/h}}.
\end{align*}
Similarly, assume $J^\ast\in \mathcal J_j$  is such that $|J^\ast|=\mathcal J_{\max}.$
By Definition \ref{config-kappa-rho}, $(I_j,\mathcal J_j)$ is a $( \varsigma,\epsilon,M,C,h)$-configuration and $|I_j|/|I|\le 1/k\rho$.  So by Lemma \ref{lenth-of-band},
\begin{align*}
  \frac{\mathcal J_{\max}}{|I|} & =\frac{|J^\ast|}{|I_j|}\frac{|I_j|}{|I|}\le\frac{Ch}{2\varsigma}\frac{1}{k\rho}=\frac{Ch}{2k\rho \varsigma}\le \frac{1}{10}.
\end{align*}
So the result follows.

(2) By (1), we have
\begin{align*}
N_r(I) \le \frac{|I|}{r} +1& \le \frac{|I|}{\mathcal J_{\min}}+1\le \frac{8k e^{C/h}}{\rho}+1\le\frac{16k e^{C/h}}{\rho}.
\end{align*}

(3)
Assume $\{I_1,\cdots,I_k\}$ are the  subintervals in Definition \ref{config-kappa-rho}(i). Since $k\le \kappa$, we have
$$
\sum_{J\in \mathcal J}\left(\frac{|J|}{|I|}\right)^\delta=\sum_{i=1}^k\sum_{J\in \mathcal J_i}\left(\frac{|J|}{|I_i|}\right)^\delta\left(\frac{|I_i|}{|I|}\right)^\delta\le \frac{\kappa}{\rho^{\delta}}\max_{1\le i\le k}\sum_{J\in \mathcal J_i}\left(\frac{|J|}{|I_i|}\right)^\delta.
$$
By Definition \ref{config-kappa-rho}(ii), each $(I_i, \mathcal J_i)$ is a $(\varsigma,\epsilon,M,C,h)$-configuration.
Thus to get \eqref{delta}, we only need to show that for a standard $(\varsigma,\epsilon, M, C,h)$-configuration  $(I,\mathcal J)$, we have
\begin{equation}\label{delta-1}
  \sum_{J\in \mathcal J}\left(\frac{|J|}{|I|}\right)^\delta\le\frac{\rho^\delta}{\kappa}.
\end{equation}

Now assume  $(I,\mathcal J)$ is a standard $(\varsigma,\epsilon, M, C,h)$-configuration. By Definition \ref{config-standard}(i), we have $|I|\ge 2\varsigma$. To show \eqref{delta-1}, we only need to show that
\begin{equation}\label{delta-2}
  S(\delta):=\sum_{J\in \mathcal J}|J|^\delta\le\frac{(2\varsigma\rho)^\delta}{\kappa}.
\end{equation}

We  split the sum in \eqref{delta-2} into three parts and estimate them separately.
Write
\begin{equation*}
  S(\delta)=\sum_{J\in \mathcal J}|J|^\delta=\sum_{i\in \mathscr I_{in}}|J_i|^\delta+\sum_{i\in \mathscr I_{out}}|J_i|^\delta+\sum_{i\in \mathscr I_{mid}}|J_i|^\delta=: S_{in}(\delta)+ S_{out}(\delta)+ S_{mid}(\delta).
\end{equation*}
By Definition \ref{config-standard}(iii) and (iv), we have
\begin{align}\label{S-in}
  S_{in}(\delta) & =\sum_{-r_1\le i\le s_1}|J_i|^\delta\le (Ch)^\delta+(-2C\log h)\left(\frac{Ch}{-\log h}\right)^\delta.
\end{align}
By Definition \ref{config-standard}(ii) and (v), we have
\begin{align}\label{S-out}
  S_{out}(\delta) & =\sum_{i\in \mathscr I_{out}}|J_i|^\delta\le 2\frac{C}{h}\exp\left(\frac{-\delta}{Ch}\right).
\end{align}
By \eqref{S-in} and \eqref{S-out} it is seen that there exists $h_1\in (0,\tilde h]$ such that for any $h\in (0,h_1]$ we have
\begin{equation}\label{est-S-in-out}
  S_{in}(\delta),\ \ \  S_{out}(\delta)\le \frac{(2\varsigma\rho)^\delta}{3\kappa}.
\end{equation}

The estimate of $S_{mid}(\delta)$ is more complex. We need to split the sum further according to different scales. We proceed as follows. Assume $L=L(h)\in \N$ is such that
\begin{equation*}
  e^{-L-1}< h\le e^{-L}.
\end{equation*}
Recall that $c_i$ is the center of $J_i.$
For any $1\le l\le L$, define
$$
\mathscr I_{mid}^l:=\{i\in \mathscr I_{mid}: e^{-l-1}<|c_i|\le e^{-l}\}.
$$
Since $e^{-L-1}<h< Mh\le|c_i|\le \epsilon<e^{-1}$, we have
$$
\mathscr I_{mid}=\bigsqcup_{l=1}^L \mathscr I_{mid}^l.
$$
By Definition \ref{config-standard}(vi), if $i\in \mathscr I_{mid}^l$, we have
\begin{align*}
  |J_i| & \le C \exp(-e^{L-l}/eC)\frac{h}{l}.
\end{align*}
Now we estimate the cardinality of $\mathscr I_{mid}^l$. If $i-1,i\in \mathscr I_{mid}^l$, then
$$
e^{-l-1}<|c_{i-1}|<|g_i|<|c_i|\le e^{-l}.
$$
By Definition \ref{config-standard}(vi), we have
$$
|G_i|\ge \frac{h}{-C\log |g_i|}\ge \frac{h}{C(l+1)}.
$$
From this we conclude that
\begin{equation*}
  \#\mathscr I_{mid}^l\le \frac{C(l+1)(e^{-l}-e^{-l-1})}{h}+2\le \frac{6Cle^{-l}}{h}.
\end{equation*}
Consequently, we have
\begin{align*}\label{S-in}
  S_{mid}^{(l)}(\delta) :&=\sum_{i\in \mathscr I_{mid}^l}|J_i|^\delta\le \frac{6Cle^{-l}}{h}
  C^\delta \exp(-e^{L-l}\delta/eC)\frac{h^\delta}{l^\delta}\\
  &=6C^{1+\delta}l^{1-\delta}e^{-l}h^{\delta-1}\exp(-e^{L-l}\delta/eC)\\
  &\le 6C^{2}le^{-l}\exp((L+1)(1-\delta)-e^{L-l}\delta/eC)\\
  &\le 6eC^{2}le^{-l}\exp(L(1-\delta)-e^{L-l}\delta/eC).
\end{align*}
If $1\le l\le L(1-\delta/2)$, then
\begin{equation*}
  S_{mid}^{(l)}(\delta)\le 6eC^{2}L\exp(L-e^{\delta L/2}\delta/eC).
\end{equation*}
If $ l> L(1-\delta/2)$, then
\begin{equation*}
  S_{mid}^{(l)}(\delta)\le 6eC^{2}Le^{-L(1-\delta/2)}\exp(L(1-\delta))=6eC^{2}Le^{-L\delta/2}.
\end{equation*}
From these, it is seen that there exists $h(\delta)\in (0,h_1)$ such that for any $h\in (0,h(\delta)]$ and any $1\le l\le L$ we have
\begin{equation}\label{est-S-mid-k}
  S_{mid}^{(l)}(\delta)\le \frac{(2\varsigma\rho)^\delta}{3L\kappa}.
\end{equation}
Consequently, we have
\begin{equation}\label{est-S-mid}
 S_{mid}(\delta)=\sum_{l=1}^{L} S_{mid}^{(l)}(\delta)\le \frac{(2\varsigma\rho)^\delta}{3\kappa}.
\end{equation}
Combine with  \eqref{est-S-in-out} and \eqref{est-S-mid}, we get \eqref{delta-2}.
\end{proof}


%

\section{ Dimension estimates for the limit set of a nested covering structure}\label{sec-nested-covering}

Motivated by the covering structure of the spectrum $\Sigma_\alpha$ obtained in \cite{HS1,HS2,HS3}, in this section, we define an abstract nested covering structure with the special configurations as their building blocks.  This nested covering structure naturally determines a limit set.  Due to the fine metrical control of the configurations,  we can obtain the upper bounds for the fractal  dimensions of the limit set.  In the next section, we will apply the result to the spectrum of critical AMO.

 \subsection{ An abstract  language}\label{sec-symbolic}\

 At first we define an abstract language, which will be used as the index set for the abstract  nested coverings.
 It is good to keep in mind that the following definitions are intended to give  a coding for a  $(k,\rho; \varsigma,\epsilon,M,C,h)$-configuration. The types are related to  the types of bands in $\tilde \B_n$ (see Paragraph  \ref{sec-struc-spec}).

 We define the alphabet for the types as
\begin{equation*}
  \mathscr T:=\{\mathbf 1, \mathbf 2\}.
\end{equation*}

Define  $\Omega_0$ as
$$
\Omega_0:=\{\emptyset_{\mathbf t}\},\ \  \ \text{where }\ \  \mathbf t=\1 \text{ or } \2.
$$
We call $\mathbf t$ the {\it type} of the word $\emptyset_{\mathbf t}$.

Next, we define $\Omega_1$ as follows. We fix the only word $w=\emptyset_{\mathbf t}\in \Omega_0$. Choose $k_w\in \N$, and for any $1\le k\le k_w$, choose  $r_w^{[k]}, s_w^{[k]}\in \N$ and write
$$
I_w^{[k]}:=\{i\in \Z:-r_w^{[k]}\le i\le s_w^{[k]}, i\ne 0\}.
$$
Define the alphabets
$$
\mathscr A_w^{[k]}:=\{i_\1^{[k]}: i\in I_w^{[k]}\}\cup \{0_\2^{[k]}\}\ \ \text{ and }\ \
\mathscr A_w:=\bigcup_{1\le k\le k_w}\mathscr A_w^{[k]}.
$$
Then define
$$
\Omega_1:=\{w e: w\in \Omega_0, e\in \mathscr A_{w}\}.
$$

%
%
 Assume $\Omega_n$ has been defined, we define $\Omega_{n+1}$ as follows.
Fix any $w\in \Omega_n$,
choose $k_w\in \N$, and for  any $1\le k\le k_w$, choose  $r_w^{[k]}, s_w^{[k]}\in \N$. We write
$$
I_w^{[k]}:=\{i\in \Z:-r_w^{[k]}\le i\le s_w^{[k]},i\ne0\}
$$
and define the alphabets
$$
\mathscr A_w^{[k]}:=\{i_\1^{[k]}: i\in I_w^{[k]}\}\cup \{0_\2^{[k]}\}\ \ \text{ and }\ \
\mathscr A_w:=\bigcup_{1\le k\le k_w}\mathscr A_w^{[k]}.
$$
We define $\Omega_{n+1}$ as
$$
\Omega_{n+1}:=\{w e: w\in \Omega_n; e\in \mathscr A_w\}.
$$
By induction, we have defined a set $\Omega_n$ for any $n\in\Z_+$. Finally we write
$$
\Omega_\ast:=\bigcup_{n\in \Z_+}\Omega_n.
$$
We call $\Omega_\ast$ a {\it language} and call any $w\in \Omega_\ast$ a {\it word}.

For any $w\in \Omega_\ast$ and any $e=i^{[k]}_{\mathbf t}\in \mathscr A_w$, we call $\mathbf t, k, i$ the {\it type}, the {\it global index}, the {\it local index} of $e$, respectively and  will use the notations
\begin{equation*}
  T_e:=\mathbf t;\ \ \ G_e:=k;\ \ \ L_e:=i.
\end{equation*}

For any $w\in \Omega_\ast$, we also introduce an order on $\mathscr A_w$ as follows: for $e,\hat e\in \mathscr A_w$ and $e\ne \hat e$,
$$
e\prec \hat e\Leftrightarrow G_e<G_{\hat e} \text{ or } \ G_e=G_{\hat e}, L_e<L_{\hat e}.
$$

For any $w=w_0\cdots w_n\in \Omega_n$, we define $T_w:=T_{w_n}$ and  say that $w$ has type $T_{w_n}.$  We also write $|w|=n$ and say that $w$ is a word of length $n.$  For any $0\le i\le n$, we  write $w|_i:=w_0\cdots w_i$ and call it the $i$-th prefix of $w.$

\subsection{ A nested covering structure coded by $\Omega_\ast$}\

Assuming  $\Omega_\ast$ is a language defined  in the previous subsection, we now describe a nested covering structure indexed by $\Omega_\ast$.

\begin{defi}
 Given a language $\Omega_\ast$.  Assume $\mathcal I(\Omega_\ast):=\{I_w: w\in \Omega_\ast\}$ is a class of compact intervals in $\R$. We say that  $\mathcal I(\Omega_\ast)$  is  a nested covering structure if
the following hold:

(i) For any $w\in \Omega_*$,  the pair $(I_w,\mathcal J_w)$ is a configuration, where $\mathcal J_w:=\{I_{we}: e\in \mathscr A_w\}.$ Moreover, for different $e,\hat e\in \mathscr A_w,$
\begin{equation*}
  I_{we}<I_{w\hat e} \Leftrightarrow e\prec \hat e.
\end{equation*}

(ii) $\max\{|I_w|: w\in \Omega_n\}\to 0$ as $n\to\infty$.

\end{defi}

Finally, if $\mathcal I(\Omega_\ast)$ is a nested covering structure, we define the limit set of  $\mathcal I(\Omega_\ast)$ as
\begin{equation}\label{def-lim-set}
  X(\Omega_\ast):=\bigcap_{n\ge0} \bigcup_{w\in \Omega_n} I_w.
\end{equation}

Assume $\mathcal I(\Omega_\ast)$ is a nested covering structure. For each $w\in \Omega_*$ and $1\le k\le k_w$, we write
$$
\mathcal J_w^{[k]}:=\{I_{we}: e\in \mathscr A_w^{[k]}\} \ \ \text{ and }\ \ I_w^{[k]}:={\rm Co}(\bigcup_{J\in \mathcal J_w^{[k]}}J),
$$
where ${\rm Co}(A)$ denotes the convex hull of $A$. Then  $(I_w^{[k]}, \mathcal J_w^{[k]}), \ 1\le k\le k_w$ are all configurations and
$$\mathcal J_w=\bigsqcup_{1\le k\le k_w}\mathcal J_w^{[k]}.$$

\subsection{Upper bounds of the dimensions for limit sets}\

Now we can state the main result of this section:

\begin{thm}\label{abstract-dim-bd}
  Assume $\kappa\in \N$ and $\rho\in (0,1)$. Assume $\varsigma, \epsilon, M,C$ satisfy \eqref{parameter}. Fix any $\delta\in (0,1)$, let $h(\delta)$ be the constant given by Lemma \ref{pre-dim-delta}(3). Assume $\Omega_\ast$ is a language  and $\mathcal I(\Omega_\ast)$ is a nested covering structure such that for any $w\in \Omega_\ast$,   there exists $h_w>0$ such that $(I_w, \mathcal J_w)$ is a $(k_w,\rho; \varsigma,\epsilon,M,C,h_w)$-configuration with
  $(I_w^{[k]}, \mathcal J_w^{[k]}), \ 1\le k\le k_w$  the related sub $(\varsigma,\epsilon,M,C,h_w)$-configurations.

  (1) If $k_w\le \kappa$ and $h_w\in (0,h(\delta)]$ for any $w\in \Omega_\ast$, then
  $$\dim_H X(\Omega_\ast)\le \delta.$$

  (2) Fix any $h'(\delta)\in (0,h(\delta)]$.  If $k_w\le \kappa$ and $h_w\in [h'(\delta),h(\delta)]$ for any $w\in \Omega_\ast$, then
  $$\overline{\dim}_B X(\Omega_\ast)\le \delta.$$
\end{thm}

\begin{proof}
  (1)   For any $n$, $\{I_w: w\in \Omega_n\}$ forms a covering of $X(\Omega_\ast)$. Moreover, by Lemma \ref{pre-dim-delta}(1) we have
\begin{equation}\label{upper-est-n-band}
|I_w|=|I_{\emptyset_{\mathbf t}}|\prod_{i=1}^n \frac{|I_{w|_i}|}{|I_{w|_{i-1}}|}\le\frac{|I_{\emptyset_{\mathbf t}}|}{10^n}\to 0.
\end{equation}

We claim that for any $n\in \N,$
\begin{equation}\label{sim-dim}
  \sum_{w\in \Omega_n}|I_w|^\delta\le |I_{\emptyset_{\mathbf t}}|^\delta.
\end{equation}

We show it by induction. When $n=0,$ it is trivially true.

Now assume the statement is true for $n-1$. We have
\begin{align*}
  \sum_{u\in \Omega_n} |I_u|^\delta & =\sum_{w\in\Omega_{n-1}}\sum_{e\in \mathscr A_w}|I_{we}|^\delta =\sum_{w\in\Omega_{n-1}}|I_w|^\delta\sum_{e\in \mathscr A_w}\frac{|I_{we}|^\delta}{|I_w|^\delta}.
\end{align*}
By the assumption,  the pair $(I_w,\{I_{we}:e\in \mathscr A_w\})$ is a $(k_w,\rho;\varsigma,\epsilon,M,C,h_w)$-configuration with $k_w\le \kappa$ and  $h_w\le h(\delta)$. By Lemma \ref{pre-dim-delta}(3), we have
$$
\sum_{e\in \mathscr A_w}\frac{|I_{we}|^\delta}{|I_w|^\delta}\le 1.
$$
So by induction hypothesis, we have
\begin{align*}
  \sum_{u\in \Omega_n} |I_u|^\delta =\sum_{w\in\Omega_{n-1}}|I_w|^\delta\sum_{e\in \mathscr A_w}\frac{|I_{we}|^\delta}{|I_w|^\delta}\le \sum_{w\in\Omega_{n-1}}|I_w|^\delta\le |I_{\emptyset_{\mathbf t}}|^\delta.
\end{align*}
By induction, the claim holds.

Then by the definition of Hausdorff measure, we have
$$
\mathscr H^\delta(X(\Omega_\ast))\le |I_{\emptyset_{\mathbf t}}|^\delta,
$$
where $\mathscr H^\delta$ denotes $\delta$-dimensional Hausdorff measure. Consequently, $\dim_H X(\Omega_\ast)\le \delta.$

(2)  Assume $0<r<|I_{\emptyset_{\mathbf t}}|$, let us estimate $N_r(X(\Omega_\ast))$. Our strategy is the following: we construct an interval  covering of $X(\Omega_\ast)$ such that each interval in this covering has approximately the size $r$, then we estimate the number of intervals.

At first, we define a subset of  $\Omega_\ast$  as follows. By \eqref{upper-est-n-band},  for any $w\in\Omega_n$ we have
$$
|I_w|\le |I_{\emptyset_{\mathbf t}}|/{10^n}.
$$

Take $N\in \N$ such that $|I_{\emptyset_{\mathbf t}}|/10^N<r.$ For each $w\in \Omega_N$, let $0\le m_w\le N$ be the integer such that
\begin{equation}\label{m-w}
  |I_{w|_{m_w}}|>r, (\mathcal J_{w|_{m_w}})_{\min}\le r.
\end{equation}
Since $|I_{\emptyset_{\mathbf t}}|>r$ and $|I_w|<r$, the number $m_w$ is well-defined and unique. Now define
$$
\mathcal U:=\{w|_{m_w}: w\in \Omega_N\}.
$$

We claim that:  $\{I_u: u\in \mathcal U\}$ is a covering of $X(\Omega_\ast)$ and
$\sum_{u\in \mathcal U}|I_u|^\delta\le |I_{\emptyset_{\mathbf t}}|^\delta.$

Indeed,
by \eqref{def-lim-set}, $\{I_w: w\in \Omega_N\}$ is a covering of $X(\Omega_\ast)$. Since for each $w\in \Omega_N$, we have $u=w|_{m_w}\in \mathcal U$ and $I_w\subset I_u$, so the first statement holds.

Notice that if $u,\hat u\in \mathcal U$ and $u\ne \hat u$, then $u$ and $\hat u$ are non-compatible, that is, neither $u\lhd \hat u$, nor $\hat u\lhd u$ (where $u\lhd v$ means that $u$ is a prefix of $v$). This is due to the uniqueness of $m_w$ in \eqref{m-w}. Now write
$$m_\ast=\min\{|u|: u\in \mathcal U\}\ \ \text{ and }\ \  m^\ast=\max\{|u|: u\in \mathcal U\}.$$ Then $0\le m_\ast\le m^\ast\le N.$ For any $m_\ast\le k\le m^\ast$, write
$$
\mathcal U_k:=\{u\in \mathcal U: |u|=k\}.
$$
Then $\mathcal U=\bigsqcup_{k=m_\ast}^{m^\ast} \mathcal U_k.$
By \eqref{sim-dim} and Lemma \ref{pre-dim-delta}(3), we have
\begin{eqnarray*}
  |I_{\emptyset_{\mathbf t}}|^\delta&\ge& \sum_{u\in \Omega_{m_\ast}}|I_u|^\delta=\sum_{u\in \mathcal U_{m_\ast}}|I_u|^\delta+\sum_{u\in \Omega_{m_\ast}\setminus\mathcal U_{m_\ast}}|I_u|^\delta\\
  &\ge&\sum_{u\in \mathcal U_{m_\ast}}|I_u|^\delta+\sum_{u\in \Omega_{m_\ast}\setminus\mathcal U_{m_\ast}}\sum_{e\in\mathscr A_u}|I_{ue}|^\delta=\sum_{u\in \mathcal U_{m_\ast}}|I_u|^\delta+\sum_{v\in \mathcal V_{m_\ast+1}}|I_v|^\delta,
\end{eqnarray*}
where $\mathcal V_{m_\ast+1}:=\{ue: u\in \Omega_{m_\ast}\setminus\mathcal U_{m_\ast},e\in\mathscr A_u\}$. Since words in $\mathcal U_{m_\ast+1}$ and $\mathcal U_{m_\ast}$ are non compatible, we must have $\mathcal U_{m_\ast+1}\subset \mathcal V_{m_\ast+1}.$ Consequently,
\begin{eqnarray*}
  |I_{\emptyset_{\mathbf t}}|^\delta&\ge& \sum_{u\in \mathcal U_{m_\ast}}|I_u|^\delta+\sum_{v\in \mathcal V_{m_\ast+1}}|I_v|^\delta\\
  &=&\sum_{u\in \mathcal U_{m_\ast}}|I_u|^\delta+\sum_{u\in \mathcal U_{m_\ast+1}}|I_u|^\delta+\sum_{v\in \mathcal V_{m_\ast+1}\setminus\mathcal U_{m_\ast+1}}|I_{v}|^\delta.
\end{eqnarray*}
We can iterate this process. Finally we get
$$
|I_{\emptyset_{\mathbf t}}|^\delta\ge\sum_{k=m_\ast}^{m^\ast} \sum_{u\in \mathcal U_{k}}|I_u|^\delta=\sum_{u\in \mathcal U}|I_u|^\delta.
$$
Hence the second statement of the claim holds.

By \eqref{m-w} and the claim, we have
$$
|I_{\emptyset_{\mathbf t}}|^\delta\ge \sum_{u\in \mathcal U}|I_u|^\delta\ge r^\delta \#\mathcal U.
$$
For each $u\in \mathcal U$, since $(I_u, \mathcal J_u)$ is a $(k_u,\rho; \varsigma,\epsilon,M,C,h_u)$-configuration with $k_u\le \kappa$ and $h'(\delta)\le h_u\le h(\delta)$, by Lemma \ref{pre-dim-delta}(2),we have
$$
N_r(I_u)\le \frac{16k_u e^{C/h_u}}{\rho}\le \frac{16\kappa e^{C/h'(\delta)}}{\rho}.
$$
Since $\{I_u: u\in \mathcal U\}$ is a covering of $X(\Omega_\ast)$, we have
$$
N_r(X(\Omega_\ast))\le \sum_{u\in \mathcal U}N_r(I_u)\le \frac{16\kappa e^{C/h'(\delta)}}{\rho}\#\mathcal U\le \frac{16\kappa e^{C/h'(\delta)}|I_{\emptyset_{\mathbf t}}|^\delta}{\rho}r^{-\delta}.
$$
Then we have
$$
\overline{\dim}_B X(\Omega_\ast)=\limsup_{r\to0}\frac{\log N_r(X(\Omega_\ast))}{-\log r}\le \delta.
$$
So the result follows.
\end{proof}

 \section{The fractal  dimensions  of the spectrum of  critical AMO}\label{sec-dim}

In this section, we apply  Theorem \ref{abstract-dim-bd} to  critical AMO and finish the proof of Theorems~\ref{main-3} and ~\ref{main-2}.

\subsection{ The nested covering  structure  of the spectrum}\

In this part, the nested covering  structure  of the spectrum established in  \cite{HS1,HS2,HS3} is essential. We summarize the  related results in Theorem \ref{key-HS-new} . The statements are extracted from all the three works and are  unified in a way  suitable for our use.   We will  make some remarks on  the connection  of the result and the works \cite{HS1,HS2,HS3}.

 \begin{rem}\label{Z-inv}
     Note that for $\alpha,\beta\in \R\setminus\Q$ with  $\alpha-\beta\in \Z$,  we have  $\Sigma_\alpha=\Sigma_\beta$.  Consequently,  both $\mathscr F_B(\delta)$ and $\mathscr F_H(\delta)$ are $\Z$-invariant.  Hence  in the rest of  this section we will  restrict $\alpha\in [0,1]\setminus \Q.$
 \end{rem}

 Assume that $\alpha\in [0,1]\setminus \mathbb{Q}$ and  has a continued fraction expansion
$$
\alpha=[a_1,a_2,a_3,\cdots].
$$
Let $G:[0,1]\setminus \mathbb{Q}\to [0,1] \setminus \mathbb{Q}$ be the Gauss map defined by
\begin{equation*}
  G(\alpha):=\{1/\alpha\}.
\end{equation*}
For any $n\in \N$,  define
\begin{equation*}
  h_n(\alpha):=2\pi[a_n,a_{n+1},a_{n+2},\cdots]=2\pi G^{n-1}(\alpha).
\end{equation*}

%
%
%
%

 \begin{thm}[\cite{HS1,HS2,HS3}] \label{key-HS-new}
 Let   $\hat m \in \Z_+$  and $\hat M \geq 2$. Then
there exist $\epsilon_0 >0$, $ \varsigma \in (0,4)$ and $\kappa\in \N$
such that,  for any $\epsilon \in (0,\epsilon_0]$, there exist $h_0>0$,
$M> C> 1$, $\rho \in (0,1)$ such that if $\alpha=[a_1,a_2,\cdots] \in
[0,1]\setminus\Q$ satisfies
$$ 1\leq a_n \leq \hat M \mbox{ for } 1\le n \leq \hat m;$$
 $q_{\hat m}(\alpha) $ is odd, where $$p_{\hat m}(\alpha)/q_{\hat m}(\alpha):=[a_1,\cdots, a_{\hat m}]$$ is the $\hat m$-th convergent of $\alpha,$ and
    $$
    h_n(\alpha)\in(0, h_0],\ \ \ (\forall n \geq \hat m +1)\,,
    $$
then the following hold:

(1) There exists a family of  disjoint  intervals $\{I_1,\cdots, I_{q_{\hat m}(\alpha)}\}$   such that
$$
\Sigma_{\alpha}=\bigcup_{1\le i\le q_{\hat m}(\alpha)}\Sigma_\alpha^{(i)},\ \ \ \text{where }\ \ \ \Sigma_\alpha^{(i)}:=\Sigma_\alpha\cap I_i.
$$

(2) For each $1\le i\le q_{\hat m}(\alpha)$, there exist a language $\Omega_\ast^{(i)}$ and a related nested covering structure $\mathcal I(\Omega_\ast^{(i)})$ such that the following hold:

(2-1) The limit set of $\mathcal I(\Omega_\ast^{(i)})$ is exactly $\Sigma_\alpha^{(i)}$,  that is,
   $$
   X(\Omega_\ast^{(i)})=\Sigma_\alpha^{(i)}.
   $$

(2-2) For each $n\ge1$ and  $w\in \Omega_{n-1}^{(i)}$,  the pair
   $(I_w,\mathcal J_w)$ is a $(k_w,\rho;\varsigma,\epsilon,M,C,h_{n+\hat m}(\alpha))$-configuration with  $k_w\le \kappa$, and with
  $$(I_w^{[k]}, \mathcal J_w^{[k]}), \ 1\le k\le k_w$$  the related sub $(\varsigma,\epsilon,M,C,h_{n+\hat m}(\alpha))$-configurations.

(2-3) If $T_w=\1$, then $k_w=1.$
\end{thm}
\begin{proof}
We have just to verify that the statement is a direct consequence of the results given in the appendices. This link was already announced in Paragraph \ref{sec-struc-spec}.  Let us consider the step 1 corresponding to the construction of $\Omega_1$. When $\hat m =0$ the statement corresponding to the properties of $\Omega_1$ is given in Proposition \ref{propB.6}. For $a_1$ large the semi-classical parameter $h= 2\pi\alpha $ is close to $2\pi /a_1$. The Harper model we are starting from is a type {\bf 1f} operator and
 we indeed get a configuration as introduced in Definition \ref{config-standard}.  Theorem \ref{theorem3} then explains what we shall do in each of the intervals with the new semi-classical parameter $h'$ close to $2\pi/a_2$.
When $\hat m>0$, the first step is different since we do not start form a type 1 operator but from an $h$-pseudo-differential system. The analysis performed in \cite{HS2} (see also Appendices A and B) leads us to a $(k,\rho;\varsigma,\epsilon,M, C,h)$ configuration
 with $k= q_{\hat m}(\alpha)$  (see Figure \ref{fig-k-rho-config} for an illustration of the configuration  when $\alpha$ is close to the rational $1/3$ (with $\hat m=1$)).  According to the statements in the appendices, we are at step 2 in each interval with a configuration with a new semi-classical parameter $h'$. Note that $k$ also appears when using  Theorem \ref{theoreme4}  relative to type $2$ operators.
 We can then iterate at any step.
\end{proof}

\begin{rem}
     The assumption that $q_{\hat m}(\alpha)$ is odd is unessential but permits to assume that the bands corresponding to the spectrum of the Harper model  for $p_{\hat m}(\alpha)/q_{\hat m}(\alpha)$ do not touch according to the result
     of \cite{VM}. Without this assumption, we should use \cite{HS2} for the description of the spectrum near the touching point (see the last line in \eqref{0.8}) and will oblige us to change at the first step the definition
      of $\tilde {\mathcal B}_1$ (see Paragraph \ref{sec-struc-spec} for the definition of $\tilde \B_1$).
    \end{rem}

%
%
%
%
%


\subsection{Arbitrarily small  Hausdorff and upper box dimensions of the spectrum}

As an application of Theorem \ref{abstract-dim-bd} and  Lemma \ref{pre-dim-delta}, we have the following estimates.

\begin{prop}\label{basic-H-B}

Let $\hat m \in \Z_+$  and $\hat M \geq 2$. Then for any $\delta>0$, there exists a constant
$L:=L(\delta,\hat m, \hat M)>0$ such that   the following hold:

(1) If
$\alpha=[a_1,a_2,\cdots,a_{\hat m},a_{\hat m+1},\cdots]$ is  such that $q_{\hat m}(\alpha)$ is odd and
   \begin{equation}\label{alpha-condi-1}
  1\le a_n\le \hat M,\ (1\le n\le \hat m); \ \ L\le a_n, \ (n\ge \hat m+1),
  \end{equation}
   then $\dim_H\Sigma_\alpha\le \delta$.

   (2) Fix any $\gamma>1$. If
$\alpha=[a_1,a_2,\cdots,a_{\hat m},a_{\hat m+1},\cdots]$ is such that $q_{\hat m}(\alpha)$ is odd and
   $$
  1\le a_n\le \hat M,\ (1\le n\le \hat m); \ \ L\le a_n\le \gamma L, \ (n\ge \hat m+1),  $$
   then $\overline{\dim}_B\Sigma_\alpha\le \delta$.
\end{prop}

\begin{proof}
(1)
  Let $\epsilon_0,\varsigma,\kappa$ be the constants in Theorem \ref{key-HS-new}.
Fix
$$\epsilon\in (0, \frac{\varsigma}{100}\wedge \epsilon_0].$$ Let $h_0>0,M> C>1,\rho\in (0,1)$ be the constants determined by Theorem \ref{key-HS-new}.

Now fix any $\delta>0$. Choose $h(\delta,\hat m,\hat M)\in (0,h_0)$ such that  Lemma \ref{pre-dim-delta} holds. Choose
$$L(\delta,\hat m,\hat M):=2\pi/h(\delta,\hat m,\hat M).$$
Now fix any $\alpha=[a_1,a_2,\cdots]$ such that \eqref{alpha-condi-1} holds.
Then for any $n\ge \hat m+1$, we have
$$
h_n(\alpha)= \frac{2\pi}{a_n+G_{n}(\alpha)}\le \frac{2\pi}{a_n}\le h(\delta,\hat m,\hat M)\le h_0.
$$
By Theorem \ref{key-HS-new}, there exists $\{I_i: 1\le i\le q_{\hat m}(\alpha)\}$ and  for any $1\le i\le q_{\hat m}(\alpha),$ there exist $\Omega_\ast^{(i)}$ and $\mathcal I(\Omega_\ast^{(i)})$ such that Theorem \ref{key-HS-new}  (2-1) and (2-2) hold. Then by Theorem \ref{abstract-dim-bd}(1), we have
$$
\dim_H \Sigma_\alpha=\dim_H \bigcup_{1\le i\le q_{\hat m}(\alpha)}\Sigma_\alpha^{(i)}=\max_{1\le i\le q_{\hat m}(\alpha)}\dim_H\Sigma_\alpha^{(i)}=\max_{1\le i\le q_{\hat m}(\alpha)}\dim_H X(\Omega_\ast^{(i)})\le \delta.
$$

(2) By the assumption, for any $n\ge \hat m+1$ we also have
$$
h_n(\alpha)= \frac{2\pi}{a_n+G_{n}(\alpha)}\ge \frac{2\pi}{a_n+1}\ge \frac{\pi}{a_n}\ge \frac{\pi}{\gamma L(\delta,\hat m,\hat M)}= \frac{h(\delta,\hat m,\hat M)}{2\gamma}.
$$
Then by Theorem \ref{abstract-dim-bd}(2), we have
$$
\overline{\dim}_B \Sigma_\alpha=\overline{\dim}_B \bigcup_{1\le i\le q_{\hat m}(\alpha)}\Sigma_\alpha^{(i)}=\max_{1\le i\le q_{\hat m}(\alpha)}\overline{\dim}_B\Sigma_\alpha^{(i)}=\max_{1\le i\le q_{\hat m}(\alpha)}\overline{\dim}_B X(\Omega_\ast^{(i)})\le \delta.
$$
\end{proof}
\subsection{Proofs of Theorem \ref{main-3} and Theorem \ref{main-2}}\

  Given Proposition \ref{basic-H-B}, the proofs are essentially the same as that of \cite{HLQZ} Theorem 1.1. The only extra issue is  the parity of $q_{\hat m}(\alpha)$.
To deal with this, we need  the following well-known fact.  We also include a very short proof.

\begin{lem}\label{odd}
  Given $\alpha\in [0,1]\setminus \Q$, assume $p_n(\alpha)/q_n(\alpha)$ is the $n$-th convergent of $\alpha$. Then for any $n\in \N$, at least one of $q_n(\alpha), q_{n+1}(\alpha)$ is odd.
\end{lem}

\begin{proof}
  Note that $q_n(\alpha)$ can be defined recursively by
  $$
  q_{-1}(\alpha)=0, \ q_0(\alpha)=1, \ q_{n+1}(\alpha)=a_{n+1}q_{n}(\alpha)+q_{n-1}(\alpha),\  (n\ge0).
  $$
  If both $q_n(\alpha)$ and $q_{n+1}(\alpha)$ are even for some $n\ge1$, then
  $$
  q_{n-1}(\alpha)=q_{n+1}(\alpha)-a_{n+1}q_{n}(\alpha)
  $$
  is also even. By induction, we get $q_0(\alpha)$ is even, a contradiction.
\end{proof}

\noindent {\bf Proof of Theorem \ref{main-3}.}
By Remark \ref{Z-inv}, we only need to show that $\mathscr F_B(\delta)\cap [0,1]$  is dense in $[0,1]$ and has positive Hausdorff dimension.

At first we show that $\mathscr F_B(\delta)\cap [0,1]$ has positive Hausdorff dimension.  Take $L=L(\delta, 2,2)$ as in  Proposition \ref{basic-H-B}. Define
$$
\mathscr F:=\{[1,2,a_3,a_4,\cdots]: L\le a_n\le 10 L, n\ge3\}.
$$
Then for any $\alpha\in \mathscr F$, we have $q_2(\alpha)=3.$ By Proposition \ref{basic-H-B},
we have $\overline \dim_B \Sigma_\alpha\le \delta.$  So $\mathscr  F\subset \mathscr F_B(\delta)\cap [0,1]$. By \cite{Go} Theorem 11, $\dim_H G^2(\mathscr  F)>0.$ By a simple computation, we know that $G^2: \mathscr F\to G^2(\mathscr F)$ is a bi-Lipschitz map. So $\dim_H \mathscr  F=\dim_H G^2(\mathscr  F)>0$.
Hence $\dim_H\mathscr F_B(\delta)\cap [0,1]>0.$

Next we show that $\mathscr F_B(\delta)\cap [0,1]$ is dense in $[0,1]$. Fix any $N\ge 2$ and take
$$
\hat L:=L(\delta, N,N)\vee L(\delta, N+1,N).
$$
Define $\mathscr F_N:=\mathscr F_N^o\cup \mathscr F_N^e$, where
\begin{align*}
  \mathscr F_N^o&:=\{\alpha\in \R\setminus \Q: a_n\le N,\ ( n\le N);  q_N(\alpha) \text{ is odd}; a_n=\hat L, \ (n\ge N+1)\}\\
  \mathscr F_N^e&:=\{\alpha\in \R\setminus \Q: a_n\le N,\ ( n\le N);  q_N(\alpha) \text{ is even}; a_{N+1}=1, a_n=\hat L, \ (n\ge N+2)\}.
\end{align*}
By Proposition \ref{basic-H-B}, $\mathscr F_N^o\subset \mathscr F_B(\delta)\cap [0,1]$. If $\alpha\in \mathscr F_N^e$, then by Lemma \ref{odd}, $q_{N+1}(\alpha)$ is odd. Then By Proposition \ref{basic-H-B} again, $\mathscr F_N^e\subset \mathscr F_B(\delta)\cap [0,1]$. So $\mathscr F_N\subset \mathscr F_B(\delta)\cap [0,1]$. Now define
$$
\widehat{\mathscr F}:=\bigcup_{N\ge2}\mathscr F_N
$$
Then $\widehat{\mathscr F}\subset \mathscr F_B(\delta)\cap [0,1]$ and by the same argument as \cite{HLQZ} Sec. 4.4, $\widehat{\mathscr F}$ is dense in $[0,1]$. So $\mathscr F_B(\delta)\cap [0,1]$ is also dense in $[0,1]$.
\hfill $\Box$

\noindent {\bf Proof of Theorem \ref{main-2}.}\ Let $\mathcal F$ and $\widehat{\mathcal F}$  be  defined as above.  Then $\widetilde F:={\mathcal F}\cup \widehat{\mathcal F}$ is dense in $[0,1]$ and of positive Hausdorff dimension.  For any $\alpha\in \widetilde{\mathcal F}$,  we have shown that $\overline{\dim}_B\Sigma_\alpha\le \delta<1$. Consequently,   $\beta(\alpha)=0$ by \cite{JZ}. By \cite{HLQZ} Theorem 1.3, we have $\dim_H \Sigma_\alpha>0.$
\hfill $\Box$

\bigskip

\noindent{\bf Acknowledgements}.
Y.-H. Qu was supported by the National Natural Science Foundation of China, No. 11790273, No. 11871098 and No. 12371090.  Q. Zhou was partially supported by National Key R\&D Program of China (2020YFA0713300) and Nankai Zhide Foundation.

 \appendix
 \section{A short reminder on the results  in \cite{HS2}}\label{sec-ap-a}
 The following theorem was proved in \cite{HS2} and sufficient for the results given in \cite{HLQZ}.

\begin{thm}
 Let $\hat m \in \Z_+$  and $M \geq 2$. There exists $\eta_0  >0$ and, for $\eta_1 \in (0,\eta_0  )$,   constants $C=C(\hat m, M,\eta_1) >0$  and $C'=C'(\hat m, M,\eta_1) >0$ such that if
 $$\alpha=[a_1,a_2,\dots] $$
  is irrational and satisfies  for some $m\leq \hat m$
 \begin{equation}\label{0.7}
 \begin{array}{ll}
 1 \leq  a_j \leq M& \mbox{ for } 0 < j \leq m\\
 a_j \geq C & \mbox{ for } j\geq m+1\,,
 \end{array}
 \end{equation}
 then $\Sigma_{\alpha}$ is contained in the union of $q_m(\alpha)$ intervals $I_\ell (h)$ ($\ell=1,\cdots, q_m(\alpha))$ in the form
 $[\gamma_\ell (h), \delta_\ell (h)]$ with
 \begin{equation}\label{0.8}
 \begin{array}{l}
   \gamma_\ell(h)\,,\, \delta_\ell (h) \in \Sigma_{\alpha}\,,\\
 \gamma_\ell < \delta_\ell \leq \gamma_{\ell +1} < \delta_{\ell +1}\,,\\
 \gamma_\ell(h) \geq \gamma_\ell  - C' h\, \mbox{ and }   \delta_\ell (h) < \delta_\ell + C'  h\,, \\
\gamma_{\ell}(h) \geq \gamma_{\ell} + \frac{1}{C'} \sqrt{h} \mbox{ if } \delta_{\ell-1} =\gamma_{\ell}\,,
 \end{array}
 \end{equation}
 where
 \begin{equation}
  \alpha^{(m)} =[a_1,\dots,a_m]= \frac{p_m(\alpha)}{q_m(\alpha)}\,,
  \end{equation}
 \begin{equation}
 h = 2\pi |\alpha - \alpha^{(m)}| \,,
 \end{equation}
 \begin{equation}\label{0.11}
 \cup_\ell [\gamma_\ell,\delta_\ell]  =\Sigma_{\alpha^{(m)}} \,,
 \end{equation}
 \begin{equation}\label{0.12}
 d(I_\ell (h), I_{\ell+1}(h))\geq \frac 1{C' }\mbox{ if } \delta_\ell \neq \gamma_{\ell +1}
 \mbox{ and } \geq \frac 1 {C'} \sqrt{h} \mbox{ if } \delta_\ell =\gamma_{\ell +1}\,.
 \end{equation}
 For each interval $I_\ell(h)$, $\Sigma_{\alpha}\cap I_\ell (h)$ can be described as living in a union of closed intervals $J_j^{(\ell)}$ of length $\neq 0$ with
 $\partial J_{j}^{(\ell)} \subset \Sigma_{\alpha}\,$, $J_{j+1}^{(\ell)}$ on the right of $J_j^{(\ell)}$ and
 \begin{equation}\label{0.13}
\frac{1}{C'}  \frac{1}{a_{m+1}} \leq d(J_j^{(\ell)}, J_{j+1}^{(\ell)} )  \leq C'  \frac{1}{\sqrt{a_{m+1}}}\,,
 \end{equation}
 \begin{equation}
| |J_0^{(\ell)}| - 2 \eta_1|  \leq  \frac{C'}{|a_{m+1}|}\,.
 \end{equation}
 The other bands have size
 \begin{equation}\label{0.16}
 \exp \left( - C(j) |a_{m+1}| \right) \mbox{ with } \frac 1 {C'} \leq C(j) \leq C'\,.
 \end{equation}
 For $j\neq 0$, if $\kappa_j^{(\ell)} $ is the affine function sending $J_j^{(\ell)}$ into  $[-2,+2]$, then
 $$
 \kappa_j^{(\ell)}(J_{j}^{(\ell)}) \cap \Sigma_{\alpha} \subset \cup_k J_{j,k}^{(\ell)}\,,
 $$
 where the $J_{j,k}^{(\ell)}$ have analogous properties to the $J_j^{(\ell)} $ with $a_{m+1}$ replaced by $a_{m+2}$ and \eqref{0.13} can be improved in the form
 \begin{equation}
\frac{1}{C'}   \frac{1}{a_{m+2}} \leq d(J_{j,k}^{(\ell)}, J_{j,k+1}^{(\ell)}) \leq C' \frac{1}{a_{m+2}}\,.
 \end{equation}
 One can then iterate indefinitely.
   \end{thm}


   \begin{rem} $\eta_1$ corresponds  with the exclusion in each interval and at each step of the renormalization of a small interval of size $\approx 2\eta_1$
    for which another analysis has to be done and which was the object of \cite{HS3}. This corresponds  to the energy $0$ for the map $(x,\xi) \mapsto 2( \cos x + \cos \xi$). This refined analysis was not need in \cite{HLQZ}.
    \end{rem}

    \begin{rem}
    The possibility of having $\delta_\ell =\gamma_{\ell +1}$ is due to the occurrence of touching bands. van Mouche \cite{VM}  has proven that it occurs only
     when $q_m$ is even and  for $\ell = \frac{ q_m} {2}$.
     These two touching bands lead to the lower bound \eqref{0.12} and  the weaker estimate in \eqref{0.13}.
     \end{rem}

     In the next appendix, we will give a more precise information in the excluded intervals appearing in the above theorem.
     \section{Extracts of Harper III}\label{sec-ap-b}

We mainly follow  the  talk of presentation \cite{HS3talk}, which was announcing \cite{HS3}. Note that, concerning the first step, the paper of Helffer--Kerdelhue \cite{HK}
can be useful for more explicit computations.

 {\it We warn that  in this section, ${\alpha}\in \Z^2$  denotes an integer vector to match the notation in \cite{HS3}, while the frequency is related to $h/2\pi$.}

  \subsection{Introduction}
  As in \cite{HS1} we are considering the Harper model
  \begin{equation}\label{eqw:(1)}
  P_0= \cos hD_x + \cos x \mbox{ in } L^2(\mathbb R), \mbox{ with } D_x=\frac 1i \frac{\partial}{\partial x}\,,
  \end{equation}
  together with perturbations of this operator under the condition that
  \begin{equation}\label{eq:(2)}
  \frac{h}{2\pi} = \frac{1}{a_1 + \frac{1}{a_2 + \frac{1}{a_3 + \ldots}}}=[a_1,a_2,a_3,\cdots]\,,
   \end{equation}
   where $a_j\in \mathbb N$, $a_j\geq C_0$. Here $C_0$ is a sufficiently large constant.  In \cite{HS1}, as recalled in Appendix A,  it has been  obtained  a partial description of the spectrum of $P_0$ by an infinite procedure of localizations in subintervals associated with affine dilations.
   Notice that
   $$
   P_0= \frac 12 (\tau_h + \tau_{-h}) + \cos x\,,
   $$
   where $\tau_h$ is the translation operator $(\tau_h u )(x)= u(x-h)\,,$ and that it can also be understood as an $h$-pseudo-differential operator of symbol $\cos \xi + \cos x$.  See the footnote 4  for the definition.

   Unfortunately, each time that an isolated part of the spectrum was localized in a fixed interval, the methods which were developed in \cite{HS1} do not permit to continue the procedure in a small neighborhood of the middle of this interval. In \cite{HS2} the results were generalized to the case when $\frac{h}{2\pi}$ is ``close" to a rational, but do not solve the previous problem. Hence, the goal of \cite{HS3talk,HS3} was to complete this description.
   If we associate with $P_0$ the symbol
   $$
   p_0(x,\xi) =\cos \xi + \cos x\,,$$
   then, for $0 < \mu \leq 2$, the surface of real energy $\mu$, i.e. the set   $ \{(x,\xi)\in \mathbb R^2| p_0(x,\xi)=\mu\}$ is decomposed as $\cup_{\alpha \in \mathbb Z^2} U_{\alpha}(\mu)$, where $U_\alpha(\mu)$ is a closed curve surrounding $2\pi \alpha$ for $0 < \mu < 2$ and which becomes $\{2\pi \alpha\}$ for $\mu=2$. For $-2 \leq \mu < 0$, we have the same type of decomposition (by just replacing $2\pi \alpha$ by $2\pi \alpha + (\pi,\pi)$).

   Following ideas by Wilkinson,  \cite{HS1} gives a localization of  the spectrum of $P_0$ in $[-2,2] \setminus [-\epsilon_0,\epsilon_0]$ (for small positive $\epsilon_0$) in a union of closed intervals of size $e^{-  1/(C h)}$, which are separated  by open intervals of size $\sim h$.
   Each interval corresponds modulo $\mathcal O(e^{-1/C h})$ to an eigenvalue of a reference problem obtained by ``filling" all the ``wells" $U_\alpha$ except one.
   Analyzing then the tunneling effect between the wells it can be proved that the study of the spectrum in each of these intervals can be reduced  to
    the analysis of the spectrum of an infinite matrix $W=(w_{\alpha,\beta})_{(\alpha,\beta)\in \mathbb Z^2\times \Z^2}$, for which the principal contribution can be  computed.

    Using the symmetry properties of $w_{\alpha,\beta}$,  $W$ is isospectral (up to a multiplicative positive constant) to the $h'$-Weyl quantification of some symbol $p(x,\xi)$ which is closed to $p_0$. Here $h'\in (0,2\pi]$ is given by
    \begin{equation}\label{defh'}
    \frac{2\pi}{h}\equiv \frac{h'}{2\pi} \mbox{ mod } \mathbb Z\,.
    \end{equation}
    This implies that
    $$
     \frac{h'}{2\pi} =[a_2,a_3,\cdots]\,.
    $$
    For the spectrum near $0$ we do not have a natural localization in wells. Actually, when $\mu$ is closed to $0$, the energy surface $p_0=\mu$ is close to the union of the lines
    $$
    \xi =\pm  x + (2k+1) \pi\,,\, k\in \mathbb Z\,,
    $$
    and the difficulty is that $p_0$ has critical points at $(k\pi,\ell \pi)$ with $ k + \ell -1\in 2 \mathbb Z$.
    One is  led to analyze the microlocal solution\footnote{ We refer to the appendices in \cite{HS3} for precise definitions. This involves
     the notion of Analytic Frequency Set and to work modulo exponentially small terms. We also refer to the more recent lectures of M. Hitrick and J. Sj\"ostrand \cite{HiSj} for an introduction of the basic definitions.} of $(P_0-\mu) u=0$ near these points.

  Let $s(0,1)$ be the segment $[(0,\pi),(\pi,0)]$,  and $s(0,j)=\mathcal H^{1-j}(s(0,1))$, where
  $$
  \mathcal H (x,\xi)=(\xi, -x)\,,$$
  and
   $s(\alpha,j) = s(0,j) + \{ 2\pi \alpha \}$\,.
   Near the middle of $s(0,1)$, the microlocal solution of $(P_0-\mu)u=0$ is unique up to a multiplicative constant. Let us denote by $u_{0,1}$ such a solution. Using the symmetries of $P_0$, we can define $u_{\alpha,j}$ a microlocal solution near $s(\alpha,j)$. Near a branching point (say for example $(0,\pi)$) the microlocal kernel of $(P_0-\mu)$ has dimension $2$ and an element of the kernel can be written as
   $$
   \begin{array}{lll}
   x_1 u_{0,1} & \mbox{ near the interior of} & s(0,1) \\
    x_3 u_{(0,1),3} & \mbox{ near the interior of} & s((0,1),3)\\
     y_2 u_{0,2} & \mbox{ near the interior of} & s(0,2)\\
      y_4 u_{(0,1),4} & \mbox{ near the interior of} & s((0,1),4)\,.
   \end{array}
   $$
 Here the parameters are  $(x_1,x_3)$ and $(y_2,y_4)$ is then determined by
   $$
 \left(  \begin{array}{c} y_2\\y_4\end{array}\right)=U \, \left(  \begin{array}{c} x_1 \\ x_3\end{array}\right)\,,
 $$
 where $U$ is a $\mu$-dependent unitary matrix whose behavior can be analyzed asymptotically.
 Now we choose a suitable family $\{f_{\alpha,j}\}$ of $L^2$ functions, such that $f_{\alpha,j}$ is microlocalized in the middle of $s(\alpha,j)$ and such that $\langle u_{\alpha,j}|f_{\alpha,j}\rangle=1$.
 Then one introduces what is called in other contexts a Grushin problem.  By introduction of  $$
 \begin{array}{rll}
 (R_+u)(\alpha,j)&:= \langle u| f_{\alpha,j}\rangle & \mbox{ for }\alpha \in \mathbb Z^2\,,\, j=1,3\,,\, u\in L^2\\
 R_-u^-&:= \sum_{\alpha \in \mathbb Z^2, j=2,4} u^-(\alpha,j) f_{\alpha,j} &\mbox{ for }  u^-\in \ell^2(\mathbb C_p)\,.
 \end{array}
 $$
 One  obtains a bijective operator
 $$
 \mathcal P_0 =\left(\begin{array}{cc} P_0-\mu&R_-\\R_+&0\end{array}\right)
 $$
 from $L^2(\mathbb R)\times \ell^2(\mathbb Z^2;\mathbb C^2_p)$ onto $L^2(\mathbb R)\times \ell^2(\mathbb Z^2;\mathbb C^2_i)$\,, where $\mathbb C^2_i$ and $\mathbb C^2_p$ are copies of $\mathbb C^2$ indexed respectively by $1,3$ and $2,4$.

 If the inverse is denoted by $\left(\begin{array}{cc} E&E_+\\E_-&E_{-+}\end{array}\right)$, then $\mu \in Sp(P_0)$ if and only if $0\in Sp(E_{-+}(\mu))$. As with $W$, one can see $E_{-+}$ as an infinite matrix $(E_{-+}(\alpha,\beta))$ where $E_{-+}(\alpha,\beta)$ is now a $2\times 2$ matrix.
 It can be  shown that $0\in Sp(E_{-+})$ if and only if $0\in Sp (P)$, where $P$ is the $h'$-quantification of a matrix-valued symbol.

 Fortunately, the analysis of these $2\times 2$ systems is rather similar to the analysis of perturbed Harper's models, and rather often, one  can indeed  reduce the analysis to the scalar case.
Nevertheless note that  the linear dependence of the spectral parameter is lost and that $P$ is no more selfadjoint.  But it can be found $h'$-pseudodifferential operators $P_1,P_2$, which are elliptic near the characteristic set of $P$, such that $P_1^*P$ and $P P_2^*$ are selfadjoint.

  \subsection{Main results}
For $h >0$ we introduce the following operators initially defined on $\mathcal S(\mathbb R)$:
\begin{enumerate}
\item
$ \mathcal F_h u (\xi ) := \frac{1}{2\pi h} \int e^{-i\xi/h} u(x)\, dx$.
\item$ \tau_1u(x)= u(x-2\pi)$, $\tau_2 u (x)= e^{2i\pi / h} u(x)$\,.
\item $T_\alpha =\tau_1^{\alpha_1} \tau_2^{\alpha_2}$ for $\alpha\in \mathbb Z^2$.
\end{enumerate}
We denote by $V$ the antilinear quantification  of the antisymplectic reflexion $(x,\xi) \mapsto \mathfrak S (x,\xi)=  (\xi,x)$. Note that.  $V^2=I$
 and that
 $$
 V= U_{\frac \pi 4} \Gamma U_{\-\frac \pi 4}\,,
 $$
 where
 $ U_t = e^{it ((hD)^2+x^2-h)/h}$  and $\Gamma$ is the complex conjugation.

In the iteration (``renormalization") procedure we will meet two types of operators.
\begin{defi} We say that the triple $(P,P_1,P_2)$ of $h$-pseudodifferential operators is of type $1$ if $P_1^*P$ and $P P_2^*$ are selfadjoint and
\begin{enumerate} \item
$[P,T_\alpha]=0\,,\, [P_j, T_\alpha]=0\,,\, \forall \alpha \in \mathbb Z^2\,,$
\item $[P,V]=0\,,\, [P_j,V]=0\,,$
\item $P \mathcal F =\mathcal F P^*$\,,\,,$P_1 \mathcal F =\mathcal F P_2^*$, $P_2 \mathcal F =\mathcal F P_1^*$\,.
\end{enumerate}
\end{defi}
This is an extension of the notion of invariant self-adjoint operators which corresponds to the case $P_1=I$ and $P_2=I$.
\begin{defi}\label{defB2}
We say that the triple $(P,P_1,P_2)$ of  $h$-pseudodifferential operators  depending on a complex parameter $\mu$ with $|\mu| < 4$ is of type\footnote{This is called ``strong type $1$" in \cite{HS3}.}  $\bf 1f$  if  it is of type $1$ for $\mu$-real and if there exists $\epsilon >0$ such that  the associated Weyl's symbols\footnote{Here we recall that if $p(x,\xi)$ is a symbol, then the associated $h$-pseudodifferential operator is defined (and denoted by $p^w(x,hD_x)$) on $\mathcal S(\mathbb R)$ by
$$  p^w(x,hD_x) u (x)= (2\pi h)^{-1} \int e^{i\frac{(x-y)\xi}{h}} p(\frac{x+y}{2},\xi) u(y) dy d\xi\,.$$}
 $p(\mu,x,\xi)$, $p_j(\mu,x,\xi)$
 are holomorphic in $\mu,x,\xi$ for $|\mu| <4$, $|\Im (x,\xi)| < \frac 1 \epsilon$ and satisfy there
 $$
 |p(x,\xi,\mu) - (\cos \xi + \cos x -\mu)|\leq \epsilon\,,\, |p_j(x,\xi,\mu)-1| \leq \epsilon\,.$$
 \end{defi}
 Note that  the Harper model (whose symbol is $\cos \xi + \cos x$) which is the first operator in the iteration procedure, is of this type.

 In this context we can define $\epsilon (P)=\epsilon (P,P_1,P_2)$ as the infimum over the $\epsilon >0$ such that the above holds.
 \begin{defi}
 We say that the triple $(P,P_1,P_2)$ of $h$-pseudodifferential operators $2\times 2$ matrices valued $L^2(\mathbb R;\mathbb C_i^2)\mapsto L^2(\mathbb R;\mathbb C_p^2)$ is of type $2$, if $P_1^* P$ and $P P_2^*$ are selfadjoint and
 \begin{enumerate} \item
$[P,T_\alpha]=0\,,\, [P_j, T_\alpha]=0\,,\, \forall \alpha \in \mathbb Z^2\,,$
\item $VP= PV T^2 $\,,\, $VP_j= P_jVT^2\,,$
\item $P \mathcal F T=\mathcal F T P^*$\,,\,$P_1 \mathcal F T=\mathcal F T P_2^*$, $P_2 \mathcal FT =\mathcal F T P_1^*$\,.
\end{enumerate}
\end{defi}

\begin{defi}\label{defB4}
 We say that the triple $(P,P_1,P_2)$ of $h$-pseudodifferential operators $2\times 2$ matrices valued and depending on a complex parameter $\mu$ with $|\mu| < 4$ is of type $\bf 2f$ if  it is of type $2$ for $\mu$-real and if there exists $\epsilon >0$ such that  the associated Weyl symbols $p(\mu,x,\xi)$, $p_j(\mu,x,\xi)$
 are holomorphic in $\mu,x,\xi$ for $|\mu| <4$, $|\Im (x,\xi)| < \frac 1 \epsilon$ and satisfy there
$$
||P(x,\xi,\mu) - P_0(a,b,x,\xi)|| \leq \epsilon\,,\, || P_j (x,\xi,\mu) - P_0(ia,ib,x,\xi)|| \leq \epsilon\,.
$$
Here
$$
P_0(a,b,x,\xi)=\left(\begin{array}{cc} b+ \bar a e^{i\xi}& \bar b + a e^{ix}\\ \bar b + a e^{-ix} & b + \bar a e^{i\xi}\end{array}
 \right)
$$
where $a,b\neq 0$ depend holomorphically on $\mu$,
$$
\bar a:= \overline{a(\bar \mu)}\,,\, \bar b:= \overline{b(\bar \mu)}\,,\, b(\mu)=b(0) (1+\beta_1(\mu)) e^{i\mu (1+\beta_2(\mu))}
$$
where the $\beta_j$ are holomorphic, $|\beta_j|\leq \epsilon$, $\beta_j$ is real for $\mu$ real.

Moreover, for $\mu \in \mathbb R$
\begin{equation}\label{eq:(4)}
|b|^2 + |a|^2 =1\,,\, |\arg b -\arg a |=\frac \pi 2\,.
\end{equation}
\end{defi}
In this context we can define $\epsilon (P)=\epsilon (P,P_1,P_2)$ as the infimum over the $\epsilon >0$ such that the above holds and
$$
C(P):= \max ( \frac{1}{|a(0)|},\frac{1}{|b(0)|})\,.
$$
For $\mu$ real, we get with the help of \eqref{eq:(4)},
\begin{equation} \label{eq:(4a)}
{\rm det} P_0 (a,b;x,\xi)= 2b\bar a \big(\frac{i}{b\bar a} \sin (2 \arg b(\mu)) + \cos \xi + \cos x\big)\,.
\end{equation}
As $\frac{i}{b\bar a}$ is real, this is a good indication that the operators of type $\bf 2f$ should behave like the operators of type $\bf 1f$. To be shorter we say that $P$ is of type $\bf 1f$ (or $\bf 2f$) if there exist $P_1,P_2$ s.t. the corresponding triple $(P,P_1,P_2)$ is of type $\bf 1f$ (or $\bf 2f$).

Finally , we define the $\mu$-spectrum of $P(\mu) $ and denote it by $\mu-Sp(P)$ the following set
$$
\mu-Sp(P)=\{\mu \in (-4,4)\mbox{  s.t. } 0 \in Sp(P (\mu))\}\,.
$$
Of course, if $P(\mu)= P-\mu$, we recover the usual notion of spectrum of $P$.

\begin{thm}\label{theorem3}
There exist $\epsilon_0 >0$ and functions $F: (0,1]\mapsto [1,+\infty)$, $h_0: (0,1]\mapsto (0,1]$, $\alpha: (0,1]^2 \mapsto (0,1]$ with $\alpha(\epsilon,h)$ tending to zero as $h$ tends to zero,  such that if $\epsilon \in(0,1]$ and if $P$ is of type $\bf 1f$ with $\epsilon (P)\leq \epsilon_0$, $0 < h \leq h_0(\epsilon)$, then
$$
\mu-Sp(P) \subset \bigcup_{j=-N_-}^{N_+}  J_j\,,
$$
where the $ J_j$ are disjoint closed intervals (placed in increasing order) such that, for each $j$, there exists a real affine map $\mathcal H_j: \mu \mapsto \mu'$, such that:
\begin{enumerate}
\item For $j\neq 0$, $$\mathcal H_j( J_j \cap \mu-Sp(P))=\mu'-Sp(Q)$$
where $Q$ is an operator of type $\bf 1f $ with $(\mu,h)$ replaced by $(\mu',h')$ and with $\epsilon (Q) \leq \epsilon$.
\item $$\mathcal H_0(J_0 \cap \,\mu-Sp(P))= \mu'-Sp(Q)\,,$$
where, for each $\mu'_0\in \mathcal H_0( J_0)$, $Q$ is an operator of type $\bf 2f$ with $(\mu,h)$ replaced by $(\mu'',h')$, where
$\mu'=\mu'_0 + \mu''$, and with
$$
\epsilon (Q)\leq \alpha (\epsilon,h)\,,\, C(Q)\leq F(\epsilon)\,.$$
For each fixed $\epsilon >0$, the width of $J_0$ satisfies
$$ \frac 1C h \leq |J_0| \leq Ch $$
and
$$
\frac 1C \frac{\log(\frac 1h)}{h}\leq  \frac{d \mathcal H_0}{d\mu} \leq C \frac{\log(\frac 1h)}{h}\,.
$$
The length of $J_j$ for $j\neq 0$ is in $[\frac 1C  e^{-1/Ch}, C \frac{h} {\log(\frac 1h)}]$ and the separation between $ J_j$ and $ J_{j+1}$ belongs to $[\frac{h}{C\log(\frac 1h)}, Ch]$.
\end{enumerate}
\end{thm}

We denote the gaps between the intervals by
$
\{ G_i: i\ne 0\}.
$
We denote the center of $J_i$ by $c_i$ and the center of $ G_i$ by $g_i$.  Then
based on our previous discussion, we summarize the results related the lengths of bands and gaps appearing in Theorem \ref{theorem3} as follows (assume $h\sim 1/a_n$):

 \begin{prop}\label{propB.6}
Let the assumptions be the same as Theorem \ref{theorem3}\footnote{Implicitly, we assume that we have shifted the spectrum of the operator in order that it takes the value $0$ at the critical point. See the beginning of  Proposition 4.4 in \cite{HS3} where some map $\mu \mapsto \mu'$ is introduced. The statement below is actually relative to some $\mu'$.}.
 Then there exist $\eta_0  >0$ and, for any $\eta_1\in (0,\eta_0]$,   $h_0>0$, $\hat C >1$
and  $M>0$ such that,  if  $0 < h \leq h_0$, one has with $\hat c =\frac{1}{\hat C}$,

1) If $J_i, G_j\subset [-4,-\eta_1  ]\cup[\eta_1  ,4]$, then
$$
\hat c \leq  - h \log |J_i| \leq  \hat C \ \ \text{ and }\ \  \ \hat c  h \leq |G_j| \leq \hat Ch\,.
$$
Moreover
$$
\frac{\hat c}h \leq \#\{i:  J_i\subset [-4,-\eta_1  ]\cup[\eta_1  ,4]\}\leq  \frac{\hat C}{h} $$

2) If $J_i, G_j\subset [-Mh,Mh]$ and $i\ne0$, then \\
$$
\frac{\hat c h}{-\log h} \leq | J_i| \leq  \frac{\hat C  h}{-\log h};\ \ \hat c \frac{h}{-\log h}|\sinh(g_j/h) | \leq  |G_j | \leq \hat C \frac{h}{-\log h}|\sinh(g_j/h) |.
$$
Moreover
$$
\hat c \log \frac 1h \leq  \#\{i:  J_i\subset [-M h,Mh]\}\leq \hat C  \log \frac 1h .
$$

3) If $J_i,  G_j\subset [-\eta_1  ,-Mh]\cup[M h,\eta_1  ]$, then\footnote{We do not try to be optimal but this is sufficient for our application.}
$$
 \hat  c  e^{-\frac{ \hat C  |c_i|}{h}}\frac{h}{-\log |c_i|} \leq |J_i| \leq  \hat C  \  e^{- \hat c  \frac{ |c_i|}{h}}\frac{h}{-\log |c_i|}  \mbox{ and }  \hat c   \frac{h}{-\log |g_j|} \leq  | G_j|\leq \hat C  \frac{h}{-\log |g_j|}.
$$

\end{prop}

\begin{thm}\label{theoreme4}There exist  functions $\tilde \epsilon_0: [1,\infty)\mapsto (0,1]$, $F: (0,1]\mapsto [1,+\infty)$, $\tilde h_0: (0,1]\times [1,\infty)\mapsto(0,1]$,
$\tilde \alpha : (0,1]\times [1,\infty)\times (0,1]\mapsto(0,1]$ with $\tilde \alpha (\epsilon,C,h)$ tending to $0$ as $h\rightarrow 0$ for fixed $(\epsilon, C)$, s.t. if $0< \epsilon \leq 1$, $C\geq 1$ and $P$ is of type $\bf 2f$ with $C(P) \leq C$, $\epsilon (P)\leq \tilde \epsilon_0(C)$, $0 < h < \tilde h_0(\epsilon,C))$, then
$$
[-3,3] \cap \mu-Sp(P) \subset \bigcup_j J_j\,,
$$
where the $J_j$ are closed disjoint intervals  (placed in increasing order), and for each $j$ there exists an affine application $\mathcal H_j:\mu \mapsto \mu'$ such that we have either $(a)$ or $(b)$ (with the same change of parameters as in the previous theorem):
\begin{itemize}
\item[(a)] $$\mathcal H_j(J_j \cap \mu-Sp(P))=\mu'-Sp(Q)$$
where $Q$ is an operator of type $\bf 1f $  with $\epsilon (Q) \leq \epsilon$.
\item[(b)]
$$\mathcal H_j(J_j \cap \,\mu-Sp(P))= \mu'-Sp(Q)\,,$$
where $Q$ is an operator of type $\bf 2f$ with
$$
\epsilon (Q)\leq \tilde \alpha (\epsilon,C, h)\,,\, C(Q)\leq F(\epsilon)\,.$$
\end{itemize}
 For each fixed $(\epsilon,C)$ there exist $E >1$ such that:\\
 \begin{itemize}
  \item In case (b) the width of $J_j$  satisfies
$$
0 <  \frac 1E h \leq | J_j| \leq E h
$$
 and
$$
0 < \frac 1E \frac{\log(\frac 1h)}{h} \leq  \frac{d \mathcal H_j}{d\mu} \leq  E  \frac{\log(\frac 1h)}{h}\,.
$$
\item
In case (a), the width of $J_j$ belongs to $[\frac 1{E}  e^{-E/h}, E \frac{h} {\log(\frac 1h)}]$ and the separation between $J_j$ and $J_{j+1}$ belongs to $[\frac{h}{E\log(\frac 1h)},Eh]$.\\
The separation between two intervals of type (b) is larger than $\frac{1}{E}$.
\end{itemize}
\end{thm}

 There are essentially no difference in the statements. Here are the small differences:
\begin{itemize}
\item We have a  finite number of ``critical values". In the (1f) case, we  only had one close to zero. The last statement of the theorem gives a uniform bound for the intervals of type (b). The theorem
gives also that the different intervals of type (b) have  uniformly comparable  size.
\item For each of these critical values, we have an interval of type (b).
\item Outside these $h$-neighborhoods of these ``critical  values" which are determined by $\sin (2 \arg b(\mu))\sim 0$ (see \eqref{eq:(4a)} and the explanations around)
 the statements for type (a) intervals are identical to the (1f) statement. We have just to replace the $c_i$ by the distance of the middle to the critical values.
\end{itemize}

Let $\epsilon_0>0$ as in Theorem \ref{theorem3} and let
$$
0 < h_1 \leq \min (h_0(\epsilon_0),\tilde h(\epsilon_0, F(\epsilon_0)))$$ small enough in order that if $h\in (0,h_1]$, we have
$$
\max (\alpha(\epsilon_0,h),\tilde \alpha(\epsilon_0, F(\epsilon_0),h) \leq \tilde \epsilon_0(F(\epsilon_0))\,.$$
(note that we can take the same $F$ in the two statements). Let $P$ be an $h$-pseudo-differential operator with $0<h \leq h_1$ satisfying
one of the assumptions $(I)$ or $(II)$
\begin{itemize}
\item[(I)] $P$ is of type $\bf 1f$ with $\epsilon (P)\leq \epsilon_0$\,,
\item[(II)] $P$ is of type $\bf 2f$ with $\epsilon (P) \leq \tilde \epsilon_0(F(\epsilon_0))$, $C(P) \leq F(\epsilon_0)$.
\end{itemize}
Then according to the above theorems, the $\mu$-spectrum is localized in the union of closed disjoint intervals and the analysis of the $\mu$-spectrum in each interval can be reduced, after an affine transformation $\mu \mapsto \mu'$ to the analysis of the $\mu'$-spectrum of $Q$, where $Q$ is an $h'$-pseudodifferential operator (with $h'$ satisfying  \eqref{defh'}) satisfying either $(I)$ or $(II)$.
If $0 <h' \leq h_1$, we can then iterate.
Taking into account the length of the intervals and their separation, we get
\begin{cor}
There exist  $\epsilon_0 >0$, $C_0 >0$ such that if
$$ \frac{h}{2\pi} = \frac{1}{a_1 + \frac{1}{a_2 + \frac{1}{a_3 + \ldots}}}=[a_1,a_2,a_3,\cdots]\,,
\mbox{ with } a_j\geq C_0$$ and if $P$ is an $h$-pseudodifferential operator of type $\bf 1f$, with $\epsilon (P) \leq \epsilon_0$, then $\mu-Sp(P)$ has measure $0$ and its complementary is dense in $\mathbb R$.
\end{cor}

\section{A slight extension}\label{appc}

In \cite{HS3} we have solved the question of the excluded middle interval appearing in \cite{HS1} but we were only consider irrational $\alpha$ sufficiently close to $0$. We show in this appendix how one should proceed when starting of an irrational sufficiently close to a rational $\frac pq$.

   The question is then to treat the middle interval at each step in order to have the same conclusion  obtained in \cite{HS3} under the stronger condition that $\hat m =0$. This extension is announced at the end of the introduction of \cite{HS3}
   in the following way:

   {\it In \cite{HS2}, the results of \cite{HS1} were extended to the case when for some $N $ we have  $|a_j| \geq C_N(a_1,..., a_N,\epsilon_0)$ for $j\geq N+1$, but still with the same incompleteness as in \cite{HS1}.
 We believe that the techniques in the present paper rather automatically lead to a more complete Cantor structure result also in that case.
 }

 The aim of this appendix is to give a few more details about what was meant by ``automatically". We recall that in the case considered in \cite{HS3} we were starting from $ \cos x+\cos hD_x $
 and that in \cite{HS2} given some rational $\frac p q =[a_1,\cdots, a_m]$ we were
  starting from $M_{p,q} (x,hD_x)$ where $M_{p,q}(x,\xi)$ was a $q \times q$ matrix with nice properties. In particular the spectrum consists in $q$ bands at most two of them touching in the middle
   defined as the image of $q$ eigenvalues $\lambda_\ell (x,\xi,p,q)$. It has been proven by van Mouche \cite{VM}  that there are no touching band when $q$ is odd.
Each of them has a unique critical value (in its middle) $\mu_\ell(p,q)$ with saddle point structure.  In \cite{HS2} (see Proposition 5.4.1 and the reminder in Appendix A), we show that outside arbitrary small neighborhoods
   $(\gamma_\ell -\epsilon_\ell,\gamma_\ell +\epsilon_\ell)$ the spectrum is contained in a union of intervals for which the spectrum is an $h'$-operator of  type {\bf 1f}  .

   What remains is to get the conclusion of Proposition 4.4 in \cite{HS3}  in $(\gamma_\ell -\epsilon_\ell,\gamma_\ell +\epsilon_\ell)$ starting from $M_{p,q} (x,hD_x)$ instead of a general  operator of type {\bf 1f}.

   Once this is proven, the proof is identical  since we are now working either with type 1 or type 2 operators. The main remark is that for this proposition we do not need the holomorphic extension in $|\Im(x,\xi)| < 1/\epsilon$ but
    only (by choosing $\epsilon_\ell$ small enough)  the analyticity is a small neighborhood of the saddle point together with the symmetries with respect to $(x,\xi)\mapsto (-\xi,x)$ and $(x,\xi)\mapsto (x,-\xi)$. Once this is observed, we can find in \cite{HS2} the statement that there exists a analytic function $f(t) = f_0(t)+ h f_1(t)$ such that the symbol of $f(M_{p,q}(x,hD_x),h)$
     can be diagonalized by block modulo an exponentially small  contribution and  the $\ell$-th eigenvalue is close to $\cos q\xi + \cos qx $. Here the main point is
     that
     $$
     {\rm Det} (M_{p,q}(x,\xi)  -\lambda) = f_{p,q} (\lambda) + (-1)^{q+1} 2 (\cos qx +\cos q\xi)
     $$
     a formula due to Chambers \cite{Ch}.
     This implies
     $$f_{p,q} (\lambda_{\ell} (x,\xi,p,q))= (-1)^{q+1} 2 (\cos qx +\cos q\xi).$$

\begin{rem}
We observe that what we get, after a diagonalization   (modulo exponentially small terms), is a reduction to a $h$-perturbation of
$$
\cos q h D_x + \cos qx.
$$
Note that this operator is unitary equivalent with
$$
\cos ((q^2h) D_x) + \cos x\,.
$$
Hence at the price of a change of initial semiclassical parameter, we are in the situation considered in \cite{HS1,HS2}.
\end{rem}
Once proven, the variant of Proposition 4.4, the results from Sections 5 and 6 in \cite{HS3} only use the conclusion of this proposition. So the only difference is that we have to consider $q$ (instead of $1$) critical values corresponding to the saddle points of  the $\lambda_{\ell,p,q} (x,\xi)$ ($\ell=1,\cdots,q$).

We insist on the fact that this is only at the first step that we have a (small) difference in the description of the spectrum.

\end{document}